\documentclass[11pt]{article}

\usepackage[margin=1.3in]{geometry}

\usepackage[dvipsnames]{xcolor} 
\usepackage[urlcolor=purple,linkcolor=blue, citecolor=Emerald, colorlinks=true, linktocpage=true]{hyperref}
\usepackage{tikz}
\usepackage{graphicx,amssymb,eucal,mathrsfs}
\usepackage{latexsym,amsmath,epsfig,epic,eepic}
\usepackage{amscd}
\usepackage{amsthm}
\usepackage{mathabx}
\usepackage{indentfirst}
\usepackage{epsf,graphicx,pgf}
\usepackage{pstricks,mathrsfs}
\usepackage{mathtools}
\usepackage{tikz-cd}
\usepackage{comment}
\usepackage{enumitem}
\usepackage{faktor}

\author{Baptiste Serraille, Ibrahim Trifa}

\newtheorem{thm}{Theorem}

\newtheorem{prop}{Proposition}[section]
\newtheorem*{prop*}{Proposition}
\newtheorem{lemma}[prop]{Lemma}
\newtheorem{lemmar}{Lemma}

\newtheorem{rk}[prop]{Remark}
\newtheorem{dfn}[prop]{Definition}

\newtheorem{coro}{Corollary}
\newtheorem{ques}{Question}

\newcommand{\Ham}{\text{Ham}}

\newcommand{\Aut}{\text{Aut}}
\newcommand{\Ker}{\text{Ker}}

\newcommand{\R}{\mathbb{R}}
\newcommand{\N}{\mathbb{N}}
\newcommand{\C}{\mathbb{C}}
\newcommand{\Z}{\mathbb{Z}}

\newcommand{\Sym}{\operatorname{Sym}}
\newcommand{\osc}{\operatorname{osc}}
\newcommand{\supp}{\operatorname{supp}}

\newcommand\restr[2]{{
		\left.\kern-\nulldelimiterspace 
		#1 
		\vphantom{|} 
		\right|_{#2} 
}}

\begin{document}

\title{On link quasimorphisms on the sphere and the equator conjecture}

\maketitle

\begin{abstract}
	Link spectral invariants were introduced by Cristofaro-Gardiner, Humilière, Mak, Seyfaddini, and Smith. They induce Hofer-Lipschitz quasimorphisms on the group of Hamiltonian diffeomorphisms of the two-dimensional sphere. We prove that some linear combinations of those quasimorphisms vanish on the stabiliser of the equator. As a consequence, at least one of the following statements holds:  there are non-trivial linear relations between the link quasimorphisms, or the space of equators of the sphere has infinite Hofer diameter.
    The proof relies on an `almost' Künneth formula in Link Floer Homology for some specific type of connected sums.
\end{abstract}

\tableofcontents

\section{Introduction}

\subsection{Link spectral invariants}

Let $(S^2,\omega)$ be the two-sphere equipped with a symplectic form, normalised so that its total area is $1$. A Lagrangian link inside the sphere $S^2$ is a finite, disjoint union of embedded circles $\underline L=L_1\cup\cdots\cup L_k\subset S^2$. The link is called $\eta$-monotone, for a non-negative real number $\eta$, if there exists a constant $\lambda>0$ such that for every connected component $B$ of $S^2\setminus \underline L$, 
\[\omega(B)+2\eta (\tau(B)-1)=\lambda,\]
where $\omega(B)$ is the symplectic area of $B$ and $\tau(B)$ its number of boundary components. $\lambda$ is called the monotonicity constant of $\underline L$, and $\eta$ the bulk parameter.

In \cite{CGHMSS22}, given such a link $\underline L$, Cristofaro-Gardiner, Humili\`ere, Mak, Seyfaddini and Smith defined a homogenised link spectral invariant $\mu_{\underline L}$, which is a map from the group $\Ham(S^2)$ of Hamiltonian diffeomorphisms of the sphere to the real line. This invariant comes from some quantitative flavour of Heegaard-Floer homology, sometimes also referred to as Link Floer homology. Link spectral invariants were first introduced to study the algebraic structure of the group of Hamiltonian homeomorphisms of a symplectic surface. They were used to prove the non-simplicity of this group for any symplectic surface \cite{CGHMSS22}, then to find new normal subgroups in \cite{CGHMSSsubleading} and \cite{MakT}.

It is showed in \cite{CGHMSS22} that the link spectral invariants are quasimorphisms, i.e. they satisfy
\[\exists D>0, \forall \varphi, \psi \in \Ham(S^2), \vert \mu_{\underline L}(\varphi \psi)-\mu_{\underline L}(\varphi)-\mu_{\underline L}(\psi)\vert \leq D.\]

This follows from work of Entov and Polterovich \cite{EP}, together with the fact that the quantum cohomology of $\C P^k$ is a field.

They also observe that the homogenised spectral invariant $\mu_{\underline L}$ only depends on $k$ (the number of components of $\underline L$) and the monotonicity constant $\lambda$. Then, for $k\geq 2$ and $\lambda\in \left[1/(k+1),1/2\right)$, one can define $\mu_{k, \lambda}$ as $\mu_{\underline L}$, where $\underline L$ is a link comprised of $k$ parallel circles in $S^2$, separating the sphere in two discs of area $\lambda$ and $k-1$ annuli of area $\frac {1-2\lambda}{k-1}$. Indeed, one can check that such a link is $\eta$-monotone (with $\eta = \frac {(k+1)\lambda-1}{2(k-1)}$), with monotonicity constant $\lambda$.

Given a family of monotone links $(\underline L_i)$ on $S^2$, with possibly distinct number of components and monotonicity constants, one can ask whether the corresponding quasimorphisms $\mu_{L_i}$ are linearly independent. If each $\underline L_i$ has a component that is disjoint from every other $\underline L_j$, then the answer is yes by the Lagrangian control property of link spectral invariants. However, the answer is unclear when the Lagrangian links overlap.

For instance, for $\lambda\in [1/3,1/2)$, consider a Lagrangian link $\underline L=L_0 \cup L_1 \cup L_2$, where $L_0$ is an equator, and $L_1$ and $L_2$ each bound a disc of area $\lambda$ in a hemisphere.
This link is monotone, and defines a quasimorphism $\mu_{3,\lambda}$. The link $L_1\cup L_2$ is also monotone, and defines a quasimorphism $\mu_{2,\lambda}$. Finally, $L_0$ is a monotone Lagrangian that we can see as a monotone link with one component, therefore it also defines a quasimorphism $\mu_{1,1/2}$.

Then, are $\mu_{3,\lambda}$, $\mu_{2,\lambda}$ and $\mu_{1,1/2}$ linearly independent?

In \cite{BFPS}, Buhovsky et al. defined a quasimorphism given by \[r_\lambda\coloneqq 3 \mu_{3,\lambda}-2\mu_{2,\lambda}-\mu_{1,1/2}.\]
They showed that it vanishes on autonomous Hamiltonian diffeomorphisms, and they wondered whether this quasimorphism vanishes identically \cite[Question 4.2]{BFPS}.

We give the beginning of an answer: let $\mathcal{S}(L_0)\subset \Ham(S^2)$ denote the stabiliser of the equator $L_0 \subset S^2$, i.e. the group of Hamiltonian diffeomorphisms of the sphere fixing $L_0$ as a set. We claim that the quasimorphism $r_\lambda$ vanishes on this stabiliser.

\begin{thm}\label{thm.main thm}
For every real number $\lambda \in [1/3,1/2)$, the homogeneous quasimorphism $r_\lambda$ given by
\[r_\lambda\coloneqq 3 \mu_{3,\lambda}-2\mu_{2,\lambda}-\mu_{1,1/2}\]
vanishes on $\mathcal{S}(L_0)$.
\end{thm}

\subsection{Lagrangian Hofer distance and the equator conjecture}

An other important feature of link spectral invariants is that they are Hofer-Lipschitz. Therefore, they turn out to be useful to study Hofer geometry. In \cite{M} and \cite{MoT}, they are used to give a lower bound to the Hofer energy of a Hamiltonian diffeomorphism fixing a link on a surface with boundary.

On the sphere, we end up with Hofer-Lipschitz quasimorphisms on $\Ham(S^2)$. As Khanevsky first pointed out in \cite{K09}, such quasimorphisms are a powerful tool when trying to prove that the space of Hamiltonian isotopies of a given Lagrangian $L_0$ has infinite Hofer diameter. Indeed, he showed that the unboundedness of the Lagrangian Hofer distance is implied by the existence of a homogeneous, non-trivial, Hofer-Lipschitz quasimorphism that vanishes on the stabiliser $\mathcal S (L_0)$. This strategy has been used to prove the unboundedness of this distance for several Lagrangians $L_0$ inside various symplectic manifolds, as in \cite{K09}, \cite{K11}, \cite{S}, \cite{T}, \cite{dawid} and \cite{zapolsky}.

However, it remains unknown whether the space of equators inside the sphere has infinite diameter. This appears as an open question in \cite{MS}, and the affirmative answer is often referred to as the `equator conjecture'.

Together with Khanevsky's argument, Theorem \ref{thm.main thm} has the following consequence:

\begin{coro}
\label{coro:alternative}
    At least one of the following statement holds:
    \begin{enumerate}[label=(\roman*)]
        \item For every real number $\lambda \in [1/3,1/2)$, the homogeneous quasimorphism $r_\lambda$ given by
    \[r_\lambda\coloneqq 3 \mu_{3,\lambda}-2\mu_{2,\lambda}-\mu_{1,1/2}\]
    vanishes identically.
        \item The equator conjecture holds.
    \end{enumerate}
\end{coro}

In particular, we recover the result of Buhovsky et al. that the quasimorphism $r_\lambda$ vanishes on autonomous Hamiltonian diffeomorphisms as a corollary of Theorem \ref{thm.main thm}.

\begin{coro}[\cite{BFPS}]\label{coro.vanish on aut}
For all $\lambda \in [1/3,1/2)$, the quasimorphism $r_\lambda$ vanishes on autonomous Hamiltonian diffeomorphisms.
\end{coro}

\begin{proof}
From \cite{HumHDR}, it is known that for all autonomous Hamiltonian diffeomorphisms $\varphi$, the sequence $d_H(L_0,\varphi^n(L_0))$ is bounded. Since $r_\lambda$ is Hofer-Lipschitz, we can proceed as in Proposition \ref{prop:Khanevsky}.
\end{proof}

\subsection{Structure of the paper}
In Section \ref{sec:preliminaries}, we fix some notations and recall some definitions and properties about the Hofer distance and quasimorphisms. Then in Section \ref{sec:Heegaard Floer}, we recall the definition of Heegaard-Floer homology following \cite{CGHMSS22}, and give more details about the holomorphic curves involved in this construction. We move on to the proof of Theorem \ref{thm.main thm} in Section \ref{sec.proof of main}, leaving out two technical lemmas that we prove in Section \ref{sec.neck-stretching and constrains}. Finally, we discuss some consequences of our result in Section \ref{sec:consequences}.

\paragraph{Acknowledgements}The authors thank their advisor Sobhan Seyfaddini, as well as Julio Sampietro Christ, Patricia Dietzsch and Oliver Edtmair for the interesting discussions about this project. I.T. would also like to thank Cheuk Yu Mak for answering some of our questions on the computation of Maslov indices of curves. B.S. would like to thank Amanda Hirschi for her support throughout this project. Both authors were partially supported by the ERC Starting Grant number 851701.

\section{Preliminaries}

In this section, we recall a number of notations and definitions. At first we recall the Hofer norm and the role that it plays in symplectic topology, then we define quasimorphisms and show how the existence of non-trivial quasimorphisms satisfying some properties has consequences for the Hofer geometry of the space of Hamiltonian isotopies of a Lagrangian submanifold.

\subsection{Notations, Hofer distance}
\label{sec:preliminaries}

Let $(M,\omega)$ be a symplectic manifold. A \textit{Hamiltonian function} on $M$ is a smooth map $H:S^1\times M \to \R$, $(t,x)\mapsto H_t(x)$. A Hamiltonian uniquely defines a time-dependent \textit{Hamiltonian vector field} $X_{H_t}$ via the formula 
\[\omega(X_{H},\cdot)=-dH.\]
The \textit{Hamiltonian diffeomorphism} generated by $H$, denoted $\varphi_H$, is defined as the time-1 map of the flow of $X_{H}$, which we denote by $\Phi^t_{H}$.

The set of Hamiltonian diffeomorphisms of $(M,\omega)$ will be denoted by $\Ham(M,\omega)$, often omitting the symplectic form $\omega$ from the notation when there is no ambiguity. The set of all Hamiltonian diffeomorphisms $\Ham(M)$ forms a group for the composition operation. Indeed, for any Hamiltonians $H$ and $K$, one can check that the diffeomorphism $\varphi_H \circ \varphi_K$ is a Hamiltonian diffeomorphism generated by the Hamiltonian $H\#K$ given by
\[(H\#K)_t(x)\coloneqq H_t(x)+K_t\left((\Phi^t_H)^{-1}(x)\right),\] 
and the inverse $\varphi_H^{-1}$ is generated by $\overline H$ given by $\overline H_t(x)\coloneqq-H_t(\Phi^t_H(x))$.

When $M$ is a closed symplectic manifold, we define the $L^{(1,\infty)}$-norm of a Hamiltonian by 
\[\lVert H \rVert_{(1,\infty)}\coloneqq \int_{S^1}\osc H_t dt=\int_{S^1}(\max H_t-\min H_t)dt.\]
The \textit{Hofer norm} of a Hamiltonian diffeomorphism $\varphi\in\Ham(M)$ is then defined as follows:
\[\lVert \varphi\rVert=\inf\{\lVert H\rVert_{(1,\infty)} \mid \, H \colon S^1\times M \to \R \text{ with }\varphi=\Phi^1_H\}.\]
One can show that the above definition gives rise to a conjugation invariant norm on the group $\Ham(M,\omega)$, we point out that its non-degeneracy is non trivial \cite{LM}. The Hofer norm thus induces a non-degenerate bi-invariant metric on $\Ham(M)$ called the \textit{Hofer distance}:
\[d(\varphi,\psi)=d(id,\varphi^{-1}\psi)\coloneqq\lVert \varphi^{-1}\psi\rVert.\]

Given a Lagrangian submanifold $L_0$ of $(M,\omega)$, let $\mathcal L(L_0)$ be the set of its Hamiltonian isotopies, i.e. \[\mathcal L(L_0)=\{\varphi(L_0) \mid\varphi\in\Ham(M)\}.\]
This set can be equipped with a distance induced by the Hofer distance on $\Ham(M)$: \[d(L,L')\coloneqq \inf\{\lVert \varphi\rVert \mid \,\varphi (L)=L'\}. \]

\subsection{Quasimorphisms and relation to the Hofer norm}

An important tool in the study of the group $\Ham(M,\omega)$ are quasimorphisms. A simple group $G$ does not admit any interesting morphism $q \colon G\to \R$ since its kernel would be trivial or the whole group; the notion of quasimorphism gives a remedy by loosening the definition of a morphism. We give a quick overview of quasimorphisms and their homogenisation.

\begin{dfn}
    A \textbf{quasimorphism} on a group $G$ is a map $q:G\to\R$ such that the function $Dq \colon G \times G \to \R$,
    \[Dq(g,h)\coloneqq |q(gh)-q(g)-q(h)|\]
    is uniformly bounded on $G\times G$.
    The \textbf{defect} of $q$ is the supremum of $Dq$ over $G \times G$. The quasimorphism $q$ is called \textbf{homogeneous} if, moreover, for any integer $n\in\Z$ and for any $g\in G$, the following identity is true \[q(g^n)=nq(g).\]
\end{dfn}

In the following proposition, we give a proof of the well-known fact that one can always build a homogeneous quasimorphism from any given quasimorphism.

\begin{prop}
    Any quasimorphism $q:G\to \R$ gives rise to a homogeneous quasimorphism by the formula 
    \[\mu(g)\coloneqq \lim\limits_{k\to\infty}\frac {q(g^k)} k \text{ .}\]
    The resulting quasimorphism $\mu$ is called the \textbf{homogenisation} of $q$.
    Moreover, if $D$ is the defect of $q$, then $|q-\mu|$ is bounded by $D$ and $\mu$ has defect at most $4D$.
\end{prop}

\begin{proof}
    We first prove that the limit is well-defined. Let $g$ be any elements in $G$, since $q$ is a quasimorphism with defect $D$, for any pair of positive integers $n$ and $m$, we get the following inequality
    \[q(g^{n+m})\leq q(g^n)+q(g^m)+D.\]
    Which we rewrite as,
    \[q(g^{n+m})+D\leq (q(g^n)+D)+(q(g^m)+D).\]
    Therefore, the sequence $(q(g^n)+D)_{n\in\N}$ is sub-additive, and by Fekete's lemma, the following limit
    \[\lim\limits_{n\to\infty}\frac {q(g^n)+D} n=\lim\limits_{n\to\infty}\frac {q(g^n)} n=\mu(g)\]
    is well defined. We go on to prove the properties of $\mu$ announced in the proposition.
    
    The fact that for all $g$, $\mu(g^{-1})=-\mu(g)$ is a consequence of the definition of $\mu$ and the fact that the quantity $\vert q(g^n)+q(g^{-n})\vert\leq q(1_G)+D$ is uniformly bounded in $n$. From the definition of $\mu$ and this last property, we immediately find that $\mu$ is homogeneous due to the following equalities
    \[\mu(g^k)=\lim\limits_{n\to\infty}\frac{q(g^{nk})}n=k\cdot \lim\limits_{n\to\infty}\frac{q(g^{nk})}{nk}=k\mu(g),\]
    that hold true for all integer $k \in \Z_{>0}$ and element $g$ of $G$. We now show that the difference between $q$ and $\mu$ is bounded by $D$. Again, by the quasimorphism property, one can obtain the estimate $|q(g^n)-nq(g)|\leq (n-1)D$ and therefore
    \begin{align}\label{eq.mu and q bounded}
        |\mu(g)-q(g)|=\left|\lim\limits_{n\to\infty}\frac {q(g^n)-nq(g)} n\right|\leq D.
    \end{align}
    The fact that $\mu$ is itself a quasimorphism of defect at most $4D$ now follows. Indeed, let $g$ and $h$ be two elements in $G$. Then, using \ref{eq.mu and q bounded},
    \[|\mu(gh)-\mu(g)-\mu(h)|\leq |q(gh)-q(g)-q(h)|+3D\leq 4D.\]
    This finishes the proof of the proposition.
\end{proof}

We come back to symplectic topology and the study of the Hofer norm. Let $L_0$ be a closed Lagrangian submanifold of a closed symplectic manifold $(M,\omega)$. Denote by $\mathcal S(L_0)$ the \textit{stabiliser} of $L_0$ inside $\Ham(M)$, i.e. the subgroup 
\[\mathcal S(L_0)\coloneqq \{\varphi\in \Ham(M)|\varphi(L_0)=L_0\},\]
of Hamiltonian diffeomorphisms that preserve $L_0$ as a set. The existence of a non-vanishing homogeneous quasimorphism on $\Ham(M,\omega)$ that is Hofer-Lipschitz and vanishes on the stabiliser implies that the space $\mathcal{L}(L_0)$ has unbounded diameter for the Hofer norm. This fact was used by Khanevsky in \cite{K09, K11}.

\begin{prop}\label{prop:Khanevsky}
    Let $\mu$ be a homogeneous Hofer-Lipschitz quasimorphism on $\Ham(M)$, that vanishes on $\mathcal S(L_0)$. For any $\varphi$ in $\Ham(M)$ such that $\mu(\varphi)\neq 0$ then $d(\varphi^n(L_0),L_0)$ grows linearly with $n$, and in particular $\mathcal L(L_0)$ has infinite Hofer diameter.
\end{prop}

\begin{proof}
Let $D$ be the defect of the quasimorphism $\mu$. We may assume without loss of generality that $\mu(\varphi) \neq 0$.

In order to show the linear growth of the sequence $d(\varphi^n(L_0),L_0)$. We need to show that the Hofer norm of any sequence of Hamiltonian diffeomorphisms $(\Psi_n)_n$ such that $\Psi_n(\varphi^n(L_0))=L_0$ has linear growth. Since the quasimorphism is Lipschitz with respect to the Hofer norm it is thus enough to show that the sequence $\vert\mu(\Psi_n)\vert$ grows linearly. By hypothesis, we know that $\mu$ vanishes on the stabiliser $\mathcal{S}(L_0)$ of $L_0$ and since $\Psi_n \circ \varphi ^n$ lies in this stabiliser, we get the following inequality
\[\vert \mu(\Psi_n)+\mu(\varphi^n)-\mu(\Psi_n \circ \varphi^n) \vert \leq D\]
which in turns yields,
\[\vert \mu(\Psi_n)+\mu(\varphi^n)\vert \leq D.\]
By the homogeneity of $\mu$ and since $\mu(\varphi)\neq 0$, the sequence $\mu(\varphi^n)$ grows linearly. Hence, the sequence $\vert\mu(\Psi_n)\vert$ also has linear growth, which completes the proof of the proposition.
\end{proof}

\section{Quantitative Heegaard Floer homology and link spectral invariants}
\label{sec:Heegaard Floer}

In order to show Theorem \ref{thm.main thm}, we need to define the link spectral invariants and go through some of their properties. We do so in this section. In Section \ref{subsec.def of spec inv}, we define the spectral invariants that we use. In Section \ref{subsec.Holo strips in Heegaard-Floer}, we review some ways to count holomorphic strips that arise in the definition of Heegaard-Floer homology. In particular, we recall the tautological correspondence as well as some Lemma of Osv\'ath and Szab\'o to count strips that correspond to cylinders under the tautological correspondence.

\subsection{Definition of the spectral invariants}\label{subsec.def of spec inv}

In this section, we recall the definition of the link spectral invariants of symplectic surfaces, following Cristofaro-Gardiner et al. \cite{CGHMSS22}. The proofs that those invariants are well-defined can be found in \cite{CGHMSS22}. In what follows $(\Sigma,\omega)$ is a closed symplectic surface, equipped with a complex structure $J_\Sigma$ such that $\omega$ is a K\"ahler form.

\begin{dfn}\label{dfn:monotoneLink}
    A Lagrangian link on $\Sigma$ is a finite disjoint union of embedded circles $\underline L=L_1\cup \ldots \cup L_k\subset \Sigma$. Let $\eta$ be a non-negative real number, and let $\underline L$ be a Lagrangian link inside $\Sigma$. Denote by $B_1, \ldots, B_s$ the connected components of $\Sigma\setminus \underline L$. Denote by $\overline{B_i}$ the closure of $B_i$, by $\tau_i$ the number of boundary components of $\overline{B_i}$ and by $A_i$ its symplectic area. Then, $\underline L$ is said to be $\eta$-monotone if the following are satisfied:
    \begin{itemize}
        \item for all $1\leq i \leq s$, $\overline{B_i}$ is planar;
        \item the quantity $A_i + 2\eta(\tau_i-1)$ does not depend on $i$.
    \end{itemize}
    This quantity is called the \textbf{monotonicity constant} of $\underline L$, and $\eta$ is called the \textbf{bulk parameter}.
\end{dfn}

\begin{rk}
    The planarity assumption can be relaxed to define the Floer homology of the link for slightly more general links, as in \cite{MoT}. However, the definition becomes more complicated, and it will not be useful for our purposes as we only consider the case of links inside the sphere.
\end{rk}

Fix for now $\underline L=L_1\cup \ldots \cup L_k\subset \Sigma$ an $\eta$ -monotone link with monotonicity constant $\lambda$. Let $X\coloneqq \Sym^k(\Sigma)$ be the $k$-th symmetric power of $\Sigma$, that is the quotient of $\Sigma^k$ by the action of the permutation group $\mathfrak S_k$ acting by permuting the coordinates. Denote by $\pi \colon \Sigma^k \to \Sym^k(\Sigma)\coloneqq \Sigma^k/ \mathfrak{S}_k$ the quotient map, and the `fat' diagonal by
\[\Delta \coloneqq \pi\left(\left\{(x_1,\ldots,x_k)\in \Sigma^k|\exists i\neq j, x_i=x_j\right\}\right).\]
For a $k$-tuple of point $(x_1, \ldots,x_k) \in \Sigma^k$ we denote $[x_1,\ldots,x_k] \coloneqq \pi(x_1,\ldots,x_k) \in \Sym^k(\Sigma)$ the corresponding point in the symmetrisation.

The form $\omega^{\oplus k}$ induces a current $\omega_X=\frac 1 {k!} \pi_*\omega^{\oplus k}$ on $X$ that is a smooth symplectic form away from $\Delta$. The complex structure $J_\Sigma^{\oplus k}$ on $\Sigma^k$ also descends to a complex structure $J_X$ on $X$.  Let $\Sym (\underline L)\coloneqq \pi(L_1\times\ldots\times L_k)$.
Let $V$ be a neighbourhood of $\Delta$ away from $\Sym (\underline L)$ (observe that since the $L_i$'s are disjoint, $\Sym (\underline L)\cap\Delta = \emptyset$). Then, following \cite[Section 7]{Perutz08}, one can find a smooth symplectic form $\omega_V$ on $X$, which coincides with $\omega_X$ outside of $V$, such that 
\[[\omega_V]=[\omega_X]=\frac 1 {k!} \pi_*[\omega^{\oplus k}].\]
Then, $\Sym (\underline L)$ becomes a Lagrangian submanifold of $(X,\omega_V)$. Moreover, it satisfies the following monotonicity property (\cite[Lemma 4.21]{CGHMSS22} for links satisfying a planarity assumption, \cite[Proposition 2.4]{MoT} for links satisfying a slightly more general condition):

\begin{prop}
    \label{prop:monotonicity}
    For all $u$ in the image of $\pi_2(X,\Sym(\underline L))\to H_2(X,\Sym(\underline L))$, we have 
    $$\omega_V(u)+\eta\Delta\cdot u =\frac \lambda 2 \mu(u),$$
    where $\mu:H_2(X,\Sym(\underline L))\to\Z$ is the Maslov class.
\end{prop}

Let $H$ be a Hamiltonian on $\Sigma$. It induces a Hamiltonian $\Sym^k (H)$ on $X$ by the formula 
\[\Sym ^k (H_t)([x_1,\ldots,x_k])=H_t(x_1)+\ldots +H_t(x_k).\] 
We can choose the neighbourhood $V$ of $\Delta$ above such that it does not intersect the trace of $\Sym (\underline L))$ under the flow of $\Sym^k(H)$ up to time 1.

\begin{dfn}\label{dfn.nearly symmetric acs}
For an almost complex structure $j$ on $\Sigma$ and an open set $V$ with
\[\Delta \subset V \subset \Sym^k(\Sigma).\]
An almost complex structure $J$ on $\Sym^k(\Sigma)$ is called $(V,\omega,j)$-\textbf{nearly symmetric} if the following are satisfied:
    \begin{itemize}
    \setlength\itemsep{-1pt}
        \item $J$ coincides with $\Sym^k(j)$ inside $V$;
        \item $J$ tames $\omega_X$ outside of $V$.
    \end{itemize}
\end{dfn}

Then, for certain well-chosen neighbourhood $V$ of the diagonal and given a generic family $(J_t)_{t\in(0,1)}$ of $(V,\omega,j)$-nearly symmetric almost complex structures, one can show, following \cite{CGHMSS22} that the Lagrangian Floer complex 
\[CF_*(\underline L,H)\coloneqq CF_*(\Sym(\underline L),\Sym^k(H),J_t)\]
is well defined, and that its homology is non-vanishing. In Section 5, instead of taking the usual small neighbourhoods $V$ of the diagonal $\Delta$, we will sometimes have to take open sets $V$ that contain some complex divisor. We recall the definition of this complex (using homological conventions instead of cohomological as in \cite{CGHMSS22}). Fix a reference point $x$ in $\Sym (\underline L)$. It induces a reference path $\gamma_x:t\mapsto \Sym\Phi^t_{H}(x)$ from $\Sym(\underline L)$ to $\Sym \Phi^1_H(\underline L)$. We define the space of paths
\[\mathcal P\coloneqq \{y:[0,1]\to X|y(0)\in\Sym\underline L, y(1)\in \Sym\Phi^1_H(\underline L)\}.\] 
For a path $y$ lying in the same connected component as $\gamma_x$, a capping is defined as an isotopy between $\gamma_x$ and $y$ in $\mathcal P$. We define the space of equivalence classes of capped paths by
\[\widetilde{\mathcal P_x}\coloneqq \left\lbrace\widehat y:[0,1]\times [0,1]\to X|\widehat y(0,\cdot)=\gamma_x(\cdot)\text{ and }\forall s\in [0,1],\widehat y(s,\cdot)\in\mathcal P\right\rbrace/\sim,\]
where $\widehat y\sim \widehat z$ if and only if $\widehat y(1,\cdot)=\widehat z(1,\cdot)$ and
\begin{align}\label{eq.equiv of cappings}
    \int_{[0,1]\times[0,1]}\widehat y^*\omega_V+\eta\widehat y \cdot\Delta=\int_{[0,1]\times[0,1]}\widehat z^*\omega_V+\eta\widehat z\cdot\Delta.
\end{align}
We define an action functional on $\widetilde{\mathcal P_x}$ by the formula 
\[\mathcal A_H([\widehat y])=\int_0^1\Sym^k(H)_t(x)dt-\int_{[0,1]\times[0,1]}\widehat y^*\omega_V-\eta\widehat y \cdot\Delta.\]

The critical points of this functional are the equivalence classes of cappings $[\widehat y]$ such that $\widehat y(1,\cdot)$ is a constant path at an intersection point between $\Sym(\underline L)$ and $\Sym(\Phi^1_H(\underline L))$.
Let $CF^0(\Sym(\underline L),\Sym^k(H))$ be the $\C$-vector space generated by those critical points.
Given a disc $u$ in $X$ with boundary on $\Sym(\underline L)$, one can glue it to a capping $\widehat y$ to obtain a new capping $u\#\widehat y$ with the same end points. By Proposition \ref{prop:monotonicity}, the equivalence class of $u\#\widehat y$ depends only on $\mu(u)$. Therefore, we get a free $\C[T^{\pm1}]$-action on $CF^0(\Sym(\underline L),\Sym^k(H),x)$ where $T$ acts on a class $[\widehat y]$ by adjoining a disc of Maslov index 2. The action of $T$ has the following consequence on the action functional
\[\mathcal A_H(T^n\cdot[\widehat y])=\mathcal A_H([\widehat y])-n\lambda.\]
Let $\Lambda\coloneqq \C[[T]][T^{-1}]$ the ring of formal Laurent series in $T$. Then, we finally define 
\[CF_*(\Sym(\underline L),\Sym^k(H))\coloneqq CF^0(\Sym(\underline L),\Sym^k(H))\otimes_{\C[T^{\pm1}]}\Lambda\,.\]

The isomorphism class of this filtered $\Lambda$-module does not depend on the choice of the reference point $x$. This can be deduced from the following lemma.

\begin{lemma}
    \label{lem:action_difference_for_cappings}
    Consider two reference points $x$ and $x'$. Then,
    \begin{enumerate}[label=(\roman*)]
        \item $\gamma_x$ and $\gamma_{x'}$ lie in the same connected component of $\mathcal P$;
        \item for any path $y$ in this connected component, and any cappings $\widehat y$ from $\gamma_x$ to $y$ and $\widehat y'$ from $\gamma_{x'}$ to $y$, we have
        \[\mathcal A^x_H([\widehat y])-\mathcal A^{x'}_H([\widehat y'])\in \lambda \Z,\]
        where for clarity's sake we added the reference point in the notation for the action.
    \end{enumerate}
\end{lemma}

\begin{proof}
    The first point is trivial: since $\Sym(\underline L)$ is path-connected, there exists a path $\alpha$ in $\Sym(\underline L)$ from $x$ to $x'$. Then, the isotopy $\beta(s,t)\coloneqq \Sym \Phi_H^t(\alpha(s))$ is a path in $\mathcal P$ from $\gamma_x$ to $\gamma_{x'}$.

    Now, let $y$ be a path in this connected component, $\widehat y$ a capping from $\gamma_x$ to $y$ and $\widehat y'$ a capping from $\gamma_{x'}$ to $y$. Let $w\coloneqq \widehat y \# (\widehat y')^{-1}$ be the capping from $\gamma_x$ to $\gamma_{x'}$ obtained by concatenating $\widehat y$ with $\widehat y'$ taken with opposite orientation. Then,

    \begin{align*}
        A^x_H([\widehat y])-\mathcal A^{x'}_H([\widehat y'])&=\int_0^1\Sym^k(H)_t(x)dt-\int_{[0,1]\times[0,1]}\widehat y^*\omega_V-\eta\widehat y \cdot\Delta\\
        &-\int_0^1\Sym^k(H)_t(x')dt+\int_{[0,1]\times[0,1]}\widehat y'^*\omega_V+\eta\widehat y' \cdot\Delta\\
        &=\int_0^1\left(\Sym^k(H)_t(x)-\Sym^k(H)_t(x')\right)dt-\int_{[0,1]\times[0,1]}w^*\omega_V-\eta w \cdot\Delta. 
    \end{align*}

    Now, for $(s,t)\in [0,1]\times[0,1]$, define $\widetilde w(s,t)\coloneqq \left(\Sym \Phi_H^t\right)^{-1}(w(s,t))$. Observe that for all $(s,t)\in[0,1]\times [0,1]$, $\widetilde w(0,t)=x$, $\widetilde w(0,t)=x$, $\widetilde w(s,0)\in \Sym(\underline L)$ and $\widetilde w(s,1)\in \Sym(\underline L)$. In particular, $[\widetilde w]$ lies in the image of $\pi_2(X,\Sym(\underline L))\to H_2(X,\Sym(\underline L))$.

    Moreover, using the fact that for all $t$ in $[0,1]$, $\left(\Sym \Phi_H^t\right)^{-1}$ is a symplectomorphism, and the definition of Hamiltonian vector fields, a computation gives
    \begin{align*}
        \int_{[0,1]\times[0,1]}(\widetilde w)^*\omega_V&=\int_{[0,1]\times[0,1]}w^*\omega_V+\int_0^1\left(\Sym^k(H)_t(\widetilde w(1,t))-\Sym^k(H)_t(\widetilde w(0,t))\right)dt\\
        &=\int_{[0,1]\times[0,1]}w^*\omega_V+\int_0^1\left(\Sym^k(H)_t(x')-\Sym^k(H)_t(x)\right)dt
    \end{align*}
    Together with the previous equation, we get
    \begin{align*}
        A^x_H([\widehat y])-\mathcal A^{x'}_H([\widehat y'])&=-\int_{[0,1]\times[0,1]}(\widetilde w)^*\omega_V-\eta w \cdot\Delta\\
        &=-\int_{[0,1]\times[0,1]}(\widetilde w)^*\omega_V-\eta \widetilde w \cdot\Delta
    \end{align*}
    where the second equality comes from the fact that for all $t$ in $[0,1]$, $\left(\Sym \Phi_H^t\right)^{-1}$ preserves the fat diagonal $\Delta$.
    Therefore, $A^x_H([\widehat y])-\mathcal A^{x'}_H([\widehat y'])$ lies in $\lambda \Z$ by Proposition \ref{prop:monotonicity}.
    
\end{proof}

\begin{rk}\label{rk:recapping_with_discs}
    Applying Lemma \ref{lem:action_difference_for_cappings} for $x=x'$, we find that the actions of two different cappings $\widehat y$ and $\widehat y'$ for the same path $y$ differ by a multiple of $\lambda$. Therefore, there exists an integer $n\in \Z$ such that $[\widehat y']=T^n\cdot[\widehat y]$, i.e one can get from any capping to another (for the same end path) by adjoining Maslov index 2 discs with boundary on $\Sym(\underline L)$ (possibly with reverse orientation).
\end{rk}

The differential $\partial$ counts holomorphic curves. To this effect, we define
\begin{equation*}
\mathcal{M}_i(\widehat{x}, \widehat{y}, H, J_t)\coloneqq 
\left\{u \colon \R \times [0,1] \to \Sym^k(\Sigma) \left\lvert
\begin{array}{lllll}
u(\R \times \{0\})\subset \Sym\underline L,\\
u(\R \times \{1\})\subset \Sym\Phi^1_H(\underline L), \\
\lim_{s\to -\infty} u(s+it)=x, \\
\lim_{s\to +\infty} u(s+it)=y\\ 
\widehat{x} \# u \sim \widehat{y}\\
    \frac{du}{ds}+J_t \frac{du}{dt}=0 \text{ and ind}(u)=i
\end{array}
\right. \right\}
\end{equation*}
The moduli space of finite energy $J_t$-holomorphic curves of Maslov index $i$ between the capped intersection points $\widehat{x}$ and $\widehat{y}$. Since $\R$ acts on the strip $\R \times [0,1]$ holomorphically by translation, we define the \textbf{unparametrized} moduli space by
\[\widehat{\mathcal{M}}_i(\widehat{x}, \widehat{y},H, J_t)=\faktor{\mathcal{M}_i(\widehat{x}, \widehat{y}, H, J_t)}{\R}\,.\]
On each generator of $CF_*(\Sym(\underline L),\Sym^k(H))$, the differential $\partial$ is defined by the formula
\[\partial [\widehat y] = \sum\limits_{[\widehat z]}n_{[\widehat y],[\widehat z]}[\widehat z],\]
where $n_{[\widehat y],[\widehat z]}$ is the algebraic count of curves in $\widehat{\mathcal{M}}_1(\widehat{y}, \widehat{z},H, J_t)$. It is then extended to the whole chain complex by linearity.

\medskip

By \cite{CGHMSS22}, we have $\partial^2=0$ and
\[HF_*(\Sym(\underline L),\Sym^k(H))\cong QH_*(\Sym(\underline L))\cong H_*(\Sym(\underline L))\otimes \Lambda,\]
where $QH_*(\Sym(\underline L))$ denotes the quantum homology of the Lagrangian $\Sym(\underline L)$, and $H_*(\Sym(\underline L))$ its Morse homology. Denote by $e$ the unit of $QH_*(\Sym(\underline L))$ (for the quantum product), and by $PSS$ the canonical isomorphism 
\[QH_*(\Sym(\underline L))\to HF_*(\Sym(\underline L),\Sym^k(H)).\]

The differential $\partial$ is action-decreasing, therefore, for $\alpha \in \R$, it gives rise to a well defined differential on the sub-vector space $CF^\alpha_*(\Sym(\underline L),\Sym^k(H))$ generated by the elements of action smaller than $\alpha$. Denote by 
\[\iota_\alpha \colon HF^\alpha_*(\Sym(\underline L),\Sym^k(H))\to HF_*(\Sym(\underline L),\Sym^k(H))\]
the map induced by the inclusion of the subcomplex. Then, we define the link spectral invariant 
$$c_{\underline L}(H)\coloneqq \dfrac 1 k \inf\left\{\alpha\in\R|PSS(e)\in Im(\iota_\alpha)\right\}\text{ .}$$
This spectral invariant is $L^{(1,\infty)}$-Lipschitz, and therefore extends to a map 
\[c_{\underline L}:C^\infty(S^1\times \Sigma)\to \R.\] 
We recall some of its properties, listed in \cite[Theorem 1.13]{CGHMSS22}.

\begin{prop}\label{prop.properties of link spec inv}
    For every $\eta$-monotone Lagrangian link $\underline L=\bigcup_{i=1}^k L_i$ on a symplectic surface $(\Sigma,\omega)$, the link spectral invariant $c_{\underline L}:C^\infty(S^1\times \Sigma)\to \R$ satisfies:
    \begin{itemize}
    \setlength\itemsep{-0.5pt}
        \item ($L^{(1,\infty)}$-Lipschitz) for any $H$, $H'$, \[\int_{S^1}\min(H_t-H'_t)dt\leq c_{\underline L}(H)-c_{\underline L}(H')\leq \int_{S^1}\max(H_t-H'_t)dt;\]
        \item (Monotonicity) if $H_t\leq H'_t$ then $c_{\underline L}(H)\leq c_{\underline L}(H')$;
        \item (Lagrangian control) if $H_t|_{L_i}=s_i(t)$ for each $i$, then \[c_{\underline L}(H)=\frac 1 k\sum_{i=1}^k\int_{S^1}s_i(t)dt;\] moreover for any $H$, \[\frac 1 k\sum_{i=1}^k\int_{S^1}\min\limits_{L_i}H_tdt\leq c_{\underline L}(H)\leq \frac 1 k\sum_{i=1}^k\int_{S^1}\max\limits_{L_i}H_tdt;\]
        \item (Support control) if $\supp(H_t)\subset \Sigma\setminus \underline L$, then $c_{\underline L}(H)=0$;
        \item (Subadditivity) $c_{\underline L}(H\# H')\leq c_{\underline L}(H)+c_{\underline L}(H')$;
        \item (Homotopy invariance) if $H,H'$ are mean-normalised and determine the same element of the universal cover $\widetilde{\Ham}(\Sigma)$, then $c_{\underline L}(H)=c_{\underline L}(H')$;
        \item (Shift) $c_{\underline L}(H+s(t))=c_{\underline L}(H)+\int_{S^1}s(t)dt$.
    \end{itemize}
\end{prop}

By the homotopy invariance and the shift property, $c_{\underline L}$ defines a map $\widetilde{\Ham}(\Sigma)\to \R$ given by 
\[c_{\underline L}(\widetilde{\varphi})\coloneqq  c_{\underline L}(H)-\int_{S^1}\int_\Sigma H_t \omega dt,\]
where $H$ is any Hamiltonian whose flow $\Phi^t_H$ represents $\widetilde\varphi$.
Then, when $\Sigma=S^2$, they show in \cite{CGHMSS22} that $c_{\underline L}:\widetilde{\Ham}(S^2)\to \R$ is a quasimorphism, whose homogenisation $\mu_{\underline L}$ descends to a quasimorphism on $\Ham(S^2)$.

\subsection{Holomorphic strips in Heegaard-Floer homology}\label{subsec.Holo strips in Heegaard-Floer}

In this section, we review how one can study the holomorphic strips that arise in the differential of the Floer complex defined earlier. In Section \ref{subsubsec.tautological correspondence}, we define the tautological correspondence. In Section \ref{subsubsec.domain}, we recall the definition of the domain of a holomorphic strip. In Section \ref{subsubsec.count sphere}, we show how the fact that the differential of some Floer complexes vanishes allows us to infer the count of some holomorphic curves. In Section \ref{subsubsec.holo punctured cylinders}, we state a lemma of Osv\'ath and Szab\'o to count holomorphic strips of index 1 that have a cylinder as a domain.

\subsubsection{Tautological correspondence}\label{subsubsec.tautological correspondence}

In this section, we recall the definition of the tautological correspondence. Let $\Sigma$ be a symplectic surface endowed with the almost complex structure $J$ and take $\underline{L}$ and $\underline{K}$ be two links with $k$ components each. Let $S$ denote the strip $\R\times [0,1]$, equipped with its standard complex structure. Each $\Sym^k(J)$-holomorphic curve 
\[u \colon (S,\partial S) \to (\Sym^k(\Sigma),\Sym^k(\underline{L})\cup \Sym^k(\underline{K}))\]
can be pull-backed by the following diagram.
\[
\begin{tikzcd}
(\widetilde{S},\partial \widetilde{S}) \arrow{r}{\widetilde{u}} \arrow[swap]{d}{\pi_{\widetilde{S}}}& (\Sigma^k, \pi^{-1}(\Sym^k(\underline{L})\cup \Sym^k(\underline{K}))) \arrow[swap]{d}{\pi}\\
(S,\partial S) \arrow{r}{u} & (\Sym^k(\Sigma),\Sym^k(\underline{L})\cup \Sym^k(\underline{K})),
\end{tikzcd}
\]
where $\pi \colon \Sigma^k \to \Sym^k(\Sigma)$ is the canonical $k!$-fold branched covering. The $k!$-fold branched covering $(\widetilde{S},\partial \widetilde{S})$ of $(S ,\partial S)$ and $\widetilde{u}$ are $\mathfrak{S}_k$-invariant. Let $\pi_1 : \Sigma^k \to \Sigma$ be the projection on the first factor. It is invariant under the action of $\mathfrak{S}_{k-1}$ by permutation of the $k-1$ last factors. Thus, the map $\pi_1 \circ \widetilde{u}$ induces a $J$-holomorphic map
\[\widehat{u} \colon (\widehat{S}, \partial \widehat{S}) \to (\Sigma,\underline{L}\cup \underline{K}),\]
where $\widehat{S}= \widetilde{S}/\mathfrak{S}_{k-1} \to S$ is a holomorphic $k$-fold branched covering of the strip $S$. Conversely, one can recover the map $u$ from the data of the $k$-fold branched covering $\widehat S$ and the map $\widehat u$: for $z\in S$, $u(z)$ is the $k$-tuple consisting of the images by $\widehat u$ of the $k$ preimages of $z$ via the covering map (where some elements are repeated if $z$ is a branching point). This construction is called the \textbf{tautological correspondence}. The tautological correspondence allows us to see a holomorphic curve in the symmetric product of a surface as a holomorphic curve on the surface itself. The drawback is the appearance of genera in $(\widehat{S}, \partial \widehat{S})$ and a more convoluted boundary decomposition. We point out that such a correspondence exists only when taking a product almost complex structure on the symmetric product. As such an almost complex structure is not generic, we might not have transversality for the holomorphic curves, and the tautological correspondence can only take us so far.

\subsubsection{Domain of holomorphic strips}\label{subsubsec.domain}

In this section, we set a notation for homotopy classes of strips and define their domain. We are again considering two transverse links $\underline L$ and $\underline K$ with $k$ components each inside a symplectic surface $\Sigma$.

For two intersection points $x$ and $y$ between $\underline L$ and $\underline K$, we denote by $\pi_2(x,y)$ the set of homotopy classes of cappings between the constant path at $x$ and the constant path at $y$. One can define a domain for such classes (and more generally for any element of $H_2(\Sym^k (\Sigma),\Sym(\underline L)\cup \Sym (\underline K))$), in the following way.

Denote by $\mathcal D_1,\ldots,\mathcal D_s$ the connected components of $\Sigma\setminus ((\underline L)\cup (\underline K))$. For $1\leq i \leq s$, fix a point $z_i$ in $\mathcal D_i$, and for $\varphi\in H_2(\Sym^k (\Sigma),\Sym(\underline L)\cup \Sym (\underline K))$, let $n_{z_i}(\varphi)$ be the intersection number between $\varphi$ and the codimension $2$ divisor $\{z_i\}\times \Sym^{k-1}(\Sigma)\subset \Sym^k (\Sigma)$.
\begin{dfn}
    Given a homology class $\varphi\in H_2(\Sym^k (\Sigma),\Sym( \underline L)\cup \Sym (\underline K))$, its domain is defined as the formal sum 
    $$\mathcal D(\varphi)=\sum\limits_{i=1}^s n_{z_i}(\varphi)\cdot \mathcal D_i \text{ .}$$
\end{dfn}

We also fix the following notation for the rest of the paper. For two intersection points $x$ and $y$ between $\underline L$ and $\underline K$, a class $\varphi\in H_2(\Sym^k (\Sigma),\Sym( \underline L)\cup \Sym (\underline K))$, and an almost complex structure $J$ on $\Sym^k(\Sigma)$, we denote by $\mathcal M_J(x,y,\varphi)$ the moduli space of Maslov index 1 $J$-holomorphic strips from $x$ to $y$ that lie in the class $\varphi$, quotiented by the $\R$-action by translation.

\subsubsection{Curve count for the link with 2 components on the sphere and minimal number of intersection points}
\label{subsubsec.count sphere}

We use the fact that the differential of a Floer complex vanishes for small perturbation of a two-component link on the sphere in order to count the holomorphic cylinders that one sees directly on the sphere.

\medskip

Let $\underline{L}=L_0 \cup L_1$ be an $\eta$-monotone link on the sphere $S^2$. We denote by $\underline{L}'=L_0'\cup L_1'$ a Hamiltonian deformation of $\underline{L}$ under a $C^\infty$-small Hamiltonian $H$. Moreover, we assume that all intersection points are transverse and that there are only 4 of them. For $i=0,1$, we denote the intersection points of $L_i$ and $L_i'$ as $p_i$ and $q_i$ in such a way that there are two strips from $p_i$ to $q_i$ lying in a small tubular neighbourhood of $L_i$ containing $L_i'$, as in Figure \ref{fig:link_two_comp}.

\begin{figure}
    \centering
    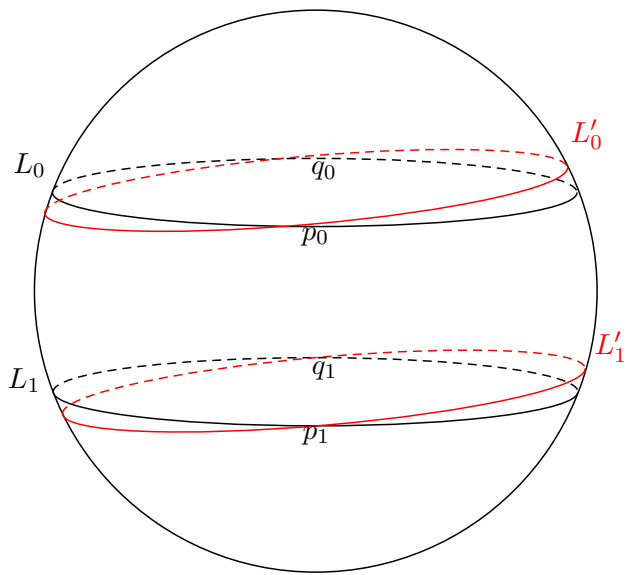
    \caption{The links $\underline L$ and $\underline L'$.}
    \label{fig:link_two_comp}
\end{figure}

We denote by $T$ the Novikov variable and by $\Lambda$ the Novikov ring. We choose a class of capping for $[p_0,p_1]$, from which we get classes of cappings for $[p_0,q_1]$, $[q_1,p_0]$ and $[q_0,q_{1}]$ by adjoining one or two strips from $p_i$ to $q_i$ lying in a small tubular neighbourhood of $L_i$. By Remark \ref{rk:recapping_with_discs}, any other class of capping for those points can be attained by recapping by Maslov index 2 discs. Then, the Floer complex $CF_*(\Sym^2(\underline{L}), \Sym^2(\underline{L}'),J)$ is generated as a $\Lambda$-module by $[p_0,p_1]$, $[p_0,q_1]$, $[q_1,p_0]$ and $[q_0,q_{1}]$ equipped with the cappings above, which we omit from the notation. The action of the Novikov variable $T$ corresponds to adjoining a Maslov index 2 disc to the capping. Since the perturbation is $C^\infty$-small, the differential $\partial$ of the Floer complex vanishes identically. In this case, it is possible to achieve transversality with a symmetric complex structure (see \cite[Proposition 3.9]{OS}).

We describe in more detail the curves that contribute to the differential when $J=\Sym^2(j)$ for a generic complex structure $j$ on $S^2$. By tautological correspondence, it is enough to study holomorphic maps from $2$-to-$1$ holomorphic coverings of the strip to $S^2$ with Lagrangian boundary conditions. The $2$-to-$1$ holomorphic coverings of the strip can be of the following type:
\begin{itemize}
\setlength\itemsep{-0.5pt}
    \item two disjoint copies of the strip;
    \item a cylinder with two punctures on each boundary component;
    \item a disc with four punctures.
\end{itemize}

\noindent Since $L_0$ does not intersect $L_1'$, there can be no map from a four-punctured disc to $S^2$ satisfying the required boundary conditions.
Therefore, a Maslov index $1$ strip contributing to the differential has to be a product of an index $1$ $j$-holomorphic strip in $S^2$ with a constant strip, or corresponds tautologically to a map from a cylinder.
The first category yields two strips from $[p_i,*]$ to $[q_i,*]$ (where $*=p_{1-i}$ or $q_{1-i}$) that cancel each other out, as well as a strip from $[q_i,*]$ to $T\cdot [p_i,*]$.
We denote by $\mathcal D_i$ and $\mathcal D_i'$ the connected components of $S^2\setminus (\underline L\cup \underline L')$ corresponding to the domains of the two strips from $p_i$ to $q_i$, by $\mathcal D_i''$ the one corresponding to the strip from $q_i$ to $p_i$, and by $\mathcal D_2$ the cylindrical connected component.
Then, a Maslov index $1$ strip tautologically corresponding to a cylinder must have for domain $\mathcal D_2+\mathcal D_i$ or $\mathcal D_2+\mathcal D_i'$ for $i=0$ or $1$. Let $\varphi_i,\varphi_i'\in \pi_2([q_i,*],[p_i,*])$ be the homotopy classes of strips whose domains are $\mathcal D_2+\mathcal D_i$ and $\mathcal D_2+\mathcal D_i'$ respectively.
Then, since $\partial = 0$, we get that $\#\mathcal M_J([q_i,*],[p_i,*],\varphi_i)+\#\mathcal M_J([q_i,*],[p_i,*],\varphi_i')=\pm 1$, with opposite sign to the curve corresponding to the domain $\mathcal D_i''$.

This count can be obtained in a more direct way by studying when punctured annuli provide holomorphic coverings of the strip (cf. \cite[Lemma 9.3 and proof of Lemma 9.4]{OS}). We now give a quick overview of this analysis.

\subsubsection{Holomorphic punctured cylinders}\label{subsubsec.holo punctured cylinders}

In this section, we state a lemma proven in \cite{OS} that counts holomorphic strips whose domain is a cylinder.

\medskip

Consider two circles $L_1$ and $L_2$ that bound an annulus inside a surface $\Sigma$. Denote $\underline L\coloneqq L_1\cup L_2$, and let $\underline L'=L'_1\cup L'_2$ be a small perturbation of $\underline L$ such that $L_i'$ intersects $L_i$ at exactly two points for $i=1,2$. Denote by $\mathcal D_1$ and $\mathcal D_2$ the two open strips delimited by $L_1$ and $L_1'$, and by $\mathcal D_3$ the annulus connected component of $\Sigma\setminus (\underline L\cup \underline L')$. We can assume that $\mathcal D_1$ and $\mathcal D_3$ are separated by an arc of $L_1$, while $\mathcal D_1$ and $\mathcal D_3$ are separated by an arc of $L_1'$. We denote by $p_i$ and $q_i$ the intersection points of $L_i$ and $L_i'$, in a way that the two small strips going from $p_i$ to $q_i$ have their lower boundary on $L_i$ and their upper boundary on $L_i'$ (see Figure \ref{fig:cylinder}).

\begin{figure}
    \centering
    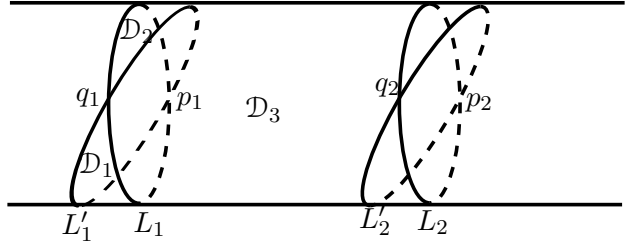
    \caption{The domains $\mathcal D_1$, $\mathcal D_2$ and $\mathcal D_3$}
    \label{fig:cylinder}
\end{figure}

Let $\varphi_1$ and $\varphi_2$ be the relative homology classes whose domains are $\mathcal D_1+\mathcal D_3$ and $\mathcal D_2+\mathcal D_3$ respectively. Let $j$ be an almost complex structure on $\Sigma$. For $x,y\in \Sym (\underline L) \cap \Sym(\underline L')$, and $\varphi \in H_2(\Sym^k (\Sigma),\Sym (\underline L)\cup \Sym (\underline L'))$, recall that $\mathcal M_{\Sym^2(j)}(x,y,\varphi)$ denotes the moduli space of Maslov index 1 $\Sym^2(j)$-holomorphic strips in $\Sym^2(\Sigma)$, in the class $\varphi$, that go from $x$ to $y$, with lower boundary on $\Sym (\underline L)$ and upper boundary on $\Sym (\underline L')$, quotiented by the $\R$-action by translation.

Then, the following is showed in the course of the proof of \cite[Lemma 9.4]{OS}.

\begin{lemma}\label{lem.count of cylinders as in OS}
    For $r_1\in\{p_1, q_1\}$, for $x=[r_1,q_2]$ and $y=[r_1,p_2]$, we have \[\#\mathcal M_{\Sym^2j}(x,y,\varphi_1)+\#\mathcal M_{\Sym^2j}(x,y,\varphi_2)=\pm 1.\]
\end{lemma}

\begin{proof}By tautological correspondence, a strip in $\mathcal M_{\Sym^2j}(x,y,\varphi_1)$ corresponds to a holomorphic $2$-fold branched cover of the strip $\widehat S$, together with a holomorphic map $v:\widehat S\to \Sigma$ whose image represents the class $\varphi_1$. Therefore, $\widehat S$ has to be a punctured annulus of the form $\{z\in \C, \frac 1 r \leq \vert z \vert \leq r\}$ with two punctures on each boundary component, for some $r\in (1,\infty)$. For $i=1,2$, denote by $\xi_i$ the arc of $\widehat S$ mapped to $L_i$ by $v$. Then, there exists a holomorphic $2$-fold branched covering $\widehat S\to S$ mapping $\xi_1$ and $\xi_2$ to the lower boundary if and only if there exists an involution of $\widehat S$ swapping $\xi_1$ and $\xi_2$, which by \cite[Lemma 9.3]{OS} happens if and only if $\xi_1$ and $\xi_2$ sweep the same angle.

Now, the image of the interior of $\widehat S$ by $v$ must be an embedded annulus, which is the complement of a slit along $L_1$ inside the annulus bounded by $L_1'$, $L_2$ and $L_2'$, as in Figure \ref{fig:slit_cylinder}.

\begin{figure}
    \centering
    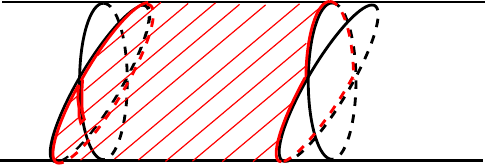
    \caption{In red, the image of the interior of $\widehat S$ by $v$. We point out to the reader that the boundary (in thick) red of $v$ has a \textit{slit}.}
    \label{fig:slit_cylinder}
\end{figure}

We can parametrize the length of the slit by $t\in [0,1)$, and denote the corresponding annulus by $A_t$. Let $f_1(t)$ be the difference of the angle swept by $\xi_1$ and the angle swept by $\xi_2$ after uniformising $A_t$ to an annulus of the same form as $\widehat S$. Then, there exists a holomorphic branched covering of the strip $\widehat S$, and a holomorphic map $v$ mapping the interior of $\widehat S$ to $A_t$ if and only if $f_1(t)=0$. Similarly for $\varphi_2$, we consider domains $B_t$ that are the complement of a slit along $L_1'$ inside the annulus bounded by $L_1$, $L_2$ and $L_2'$, and build a function $f_2(t)$ such that there exists a corresponding pair $(\widehat S,v)$ if and only if $f_2(t)=0$.

Finally, they explain in the proof of \cite[Lemma 9.4]{OS} that when the links and almost complex structure are generic, the two functions $f_1$ and $f_2$ are continuous, with $f_1(0)<0$, $f_2(0)>0$, and
\[\lim\limits_{t\to 1}f_1(t)=\lim\limits_{t\to 1}f_2(t)\neq 0,\]
which implies that the signed number of zeros of $f_1$ minus the signed count of zeros of $f_2$ is $\pm 1$, which by the previous discussion agrees with 
\[\#\mathcal M_{\Sym^2j}(x,y,\varphi_1)+\#\mathcal M_{\Sym^2j}(x,y,\varphi_2).\]
\end{proof}


\section{Proof of the main theorem}\label{sec.proof of main}

In this section, we give a proof of Theorem \ref{thm.main thm} reducing it to the proof of two technical lemmas of Section \ref{sec.neck-stretching and constrains}. The section is divided into different subsections. In Section \ref{subsec.setting and notations} we set notations that are used throughout Sections \ref{sec.proof of main} and \ref{sec.neck-stretching and constrains}, then give the proofs of Theorem \ref{thm.main thm} and Proposition \ref{prop.reducing to comp supp Ham}. In Section \ref{subsec.proof of prop 4.4}, we prove Proposition \ref{prop.eq of spec inv on diff spheres} that relates, for some special Hamiltonian diffeomorphisms, the spectral invariants from two links with two components and same monotonicity constant but on spheres of different areas. In Section \ref{subsec.almost Künneth formula}, we prove an almost Künneth formula for Link Floer homology that allows us to derive some inequality on the spectral invariants. After homogenising, this inequality is enough for us to prove the equality of Theorem \ref{thm.main thm}.

\subsection{Prolegomena}\label{subsec.setting and notations}

In this section, we fix some notations that we use in our proof of Theorem \ref{thm.main thm}. The real number $\lambda \in [1/3,1/2)$ as in Theorem \ref{thm.main thm} will be fixed from now on. 

\noindent In what follows, we denote $\mathscr{S}_3$ the sphere $S^2(1)$ of area 1 endowed with a Lagrangian link $\underline{L}$ composed of three Lagrangians $L_0$, $L_1$ and $L_2$. The first Lagrangian $L_0$ is an equator while $L_1$ and $L_2$ are two Lagrangians that bound each a disc of area $\lambda$ in the northern hemisphere and southern hemisphere respectively, see the left side of Figure \ref{fig:sphere_maps}. In general, the northern hemisphere is the open disc bounded by the Lagrangian $L_0$ represented at the top of the figures. Vice versa for the southern hemisphere.

\noindent This implies that the link $\underline{L}$ is an $\eta$-monotone link with $\eta=(4\lambda-1)/4$ and monotonicity constant $\lambda$. Moreover, one can see that since $\lambda\geq1/3$, the 2-component link $L_1 \cup L_2$ is an $\eta'$-monotone link with $\eta'=(3\lambda-1)/2$ and monotonicity constant $\lambda$. We can now define three spectral invariants $c_1$, $c_2$ and $c_3$ on the space of smooth Hamiltonians on the sphere as follows:
\begin{itemize}
\setlength\itemsep{-0.5pt}
    \item[$\bullet$] The spectral invariant $c_1$ comes from the link with one component $L_0$, (whose homogenisation is the classical Calabi quasimorphism defined by Entov and Polterovich \cite{EP});
    \item[$\bullet$] The spectral invariant $c_2$ comes from the link with 2 components $L_1 \cup L_2$;
    \item[$\bullet$] Finally, the spectral invariant $c_3$ comes from the link with 3 components $L_0 \cup L_1 \cup L_2$.
\end{itemize}
The homogenisations of the spectral invariants $c_1$, $c_2$ and $c_3$ are, by definition, $\mu_{1,1/2}$, $\mu_{2,\lambda}$ and $\mu_{3,\lambda}$. In other words, the quasimorphism $r_\lambda$ of Theorem $\ref{thm.main thm}$ is the homogenisation of the quasimorphism $3c_3-2c_2-c_1$. It is well-known from \cite{CGHMSS22} that $r_\lambda$ descends to a homogeneous quasimorphism on $\Ham(S^2(1))$. Let us further motivate the definition of this quasimorphism. In fact, from the Lagrangian control property, the quasimorphism $3c_3-2c_2-c_1$ vanishes for every Hamiltonian $H$ on the sphere that is constant along $L_0$, $L_1$ and $L_2$ at all times. Hence, the same property holds for the homogenisation and $r_\lambda(H)=0$ for any such $H$. By some clever manipulations of the Reeb tree of an autonomous Hamiltonian, Buhovsky et al. \cite{BFPS} managed to prove that the quasimorphisms $r_\lambda$ vanish on all autonomous Hamiltonian diffeomorphisms by bringing this last case to the case of Hamiltonians that are constant along $L_0$, $L_1$ and $L_2$.

\medskip

In order to prove Theorem \ref{thm.main thm}, we will prove the following proposition.

\begin{prop}\label{prop.reducing to comp supp Ham}
The quasimorphism $r_\lambda$ vanishes on Hamiltonian diffeomorphisms supported in the interior of the northern hemisphere.
\end{prop}

Note that for a Hamiltonian $H$ as in Proposition \ref{prop.reducing to comp supp Ham}, $\mu_{1,1/2}=0$, that is, we actually need to show that $3\mu_{3,B}=2\mu_{2,B}$. Theorem \ref{thm.main thm} becomes a corollary of this proposition.

\begin{proof}[Proof of Theorem \ref{thm.main thm}]
Let $\varphi \in \Ham(\mathscr{S}_3)$ be a Hamiltonian diffeomorphism that fixes $L_0$ as a set, we want to show that $\lvert r_\lambda(\varphi) \rvert$ is bounded by some constant $C>0$ that does not depend on $\varphi$. Then, since $r_\lambda$ is homogeneous and for all $n \in \N$, $\varphi^n$ also preserves $L_0$, we would have 
\[\lvert n r_\lambda(\varphi)\rvert =  \vert r_\lambda(\varphi^n)\rvert \leq C,\]
which in turn would imply that $r_\lambda(\varphi)=0$.

\medskip

Let $R$ be any Hamiltonian diffeomorphism of $\mathscr{S}_3$ that sends $L_0$ to $L_0$ but reverses its orientation (take, for instance, a solid rotation of the sphere of angle $\pi$ around an axis contained in the equator plane). Then, there exists an integer $i$ equal to 0 or 1 such that $R^i \circ \varphi$ not only maps $L_0$ to itself, but also preserves its orientation. For a small $\varepsilon >0$, we can find $\psi_\varepsilon \in \Ham(\mathscr{S}_3)$ such that its Hofer norm is bounded by $\varepsilon$ and $\psi_\varepsilon \circ R^i \circ \varphi$ is equal to the identity in some neighbourhood of $L_0$. In this way the diffeomorphism $\psi_\varepsilon \circ R^i \circ \varphi$ is the composition of two diffeomorphisms $\varphi_N$ and $\varphi_S$ supported, respectively, in the northern hemisphere and the southern hemisphere. Now any area-preserving diffeomorphism of the disc is a Hamiltonian diffeomorphism of the disc and in particular of the sphere. This implies, together with Proposition \ref{prop.reducing to comp supp Ham} that $r_\lambda(\varphi_N)=r_\lambda(\varphi_S)=0$. Now,
\[\lvert r_\lambda(\varphi)\rvert=\lvert r_\lambda(R^{-i}\circ \psi_\varepsilon^{-1}\circ \varphi_N \circ \varphi_S)\lvert \leq 3D+ \lvert r_\lambda(R)\rvert +\varepsilon.\]
This is the constant bound that we wanted. This inequality finishes the proof of Theorem~\ref{thm.main thm}.
\end{proof}

We now outline the proof of Proposition \ref{prop.reducing to comp supp Ham} and set some more notations on top of the ones at the beginning of the section. Our goal is to relate the value of the spectral invariants $c_3$ on $\Ham(\mathscr{S}_3)$ with two other spectral invariants obtained by ``splitting" the sphere $\mathscr{S}_3$ between the Lagrangians $L_0$ and $L_2$ to obtain two spheres $\mathscr{S}_2$ and $\mathscr{S}_1$.

\medskip

More precisely, on the one hand, let $\mathscr{S}_2$ be the sphere $S^2(1/2+\lambda)$ endowed with the two Lagrangians $L_1'$ and $L_0'$ bounding each a disc of area $\lambda$. For the convenience of the reader and ours, we moreover assume that $L_1'$ is in the northern hemisphere and $L_0'$ is in the southern hemisphere. On the other hand, the sphere $\mathscr{S}_1$ is the sphere $S^2(2\lambda)$ endowed with one Lagrangian $L_2'$ that is the equator of this sphere. We point out that the links with $1$, $2$ and $3$ components on the spheres $\mathscr{S}_1$, $\mathscr{S}_2$ and $\mathscr{S}_3$, respectively, have the same monotonicity constant $\lambda$. The direct consequence of this fact is that the formal variable $T$ of the Novikov ring acts on their respective Lagrangian Floer complexes by decreasing the action by $\lambda$. This will play a crucial role in what follows.

\noindent Let $D_N \subset \mathscr{S}_3$ denote the northern hemisphere, that is the open disc of area $1/2$ bounded by $L_0$ and that contains $L_1$ and let $D_2\subset \mathscr{S}_3$ be the open disc of area $\lambda$ bounded by $L_2$. Moreover, we pick four symplectomorphic neighbourhoods $\mathcal{N}(L_0)$, $\mathcal{N}(L_2)$, $\mathcal{N}(L_0')$ and $\mathcal{N}(L_2')$ of $L_0$, $L_2$, $L_0'$ and $L_2'$ respectively as well as a symplectomorphism $f \colon \mathcal{N}(L_0) \to \mathcal{N}(L_2)$ that sends $L_0$ to $L_2$. We take a smooth area-preserving embedding
\[\psi \colon D_N \cup \mathcal{N}(L_0) \to \mathscr{S}_2\]
that maps $L_0$ to $L_0'$, $L_1$ to $L_1'$ and $\mathcal{N}(L_0)$ to $\mathcal{N}(L_0')$ as well as a smooth area-preserving embedding 
\[\theta \colon D_2 \cup \mathcal{N}(L_2) \to \mathscr{S}_1\]
that maps $L_2$ to $L_2'$ and $\mathcal{N}(L_2)$ to $\mathcal{N}(L_2')$. These maps allow us to mirror any construction that we make on $\mathscr{S}_3$ to constructions on $\mathscr{S}_2$ and $\mathscr{S}_1$.

\begin{figure}[ht]
    \centering
    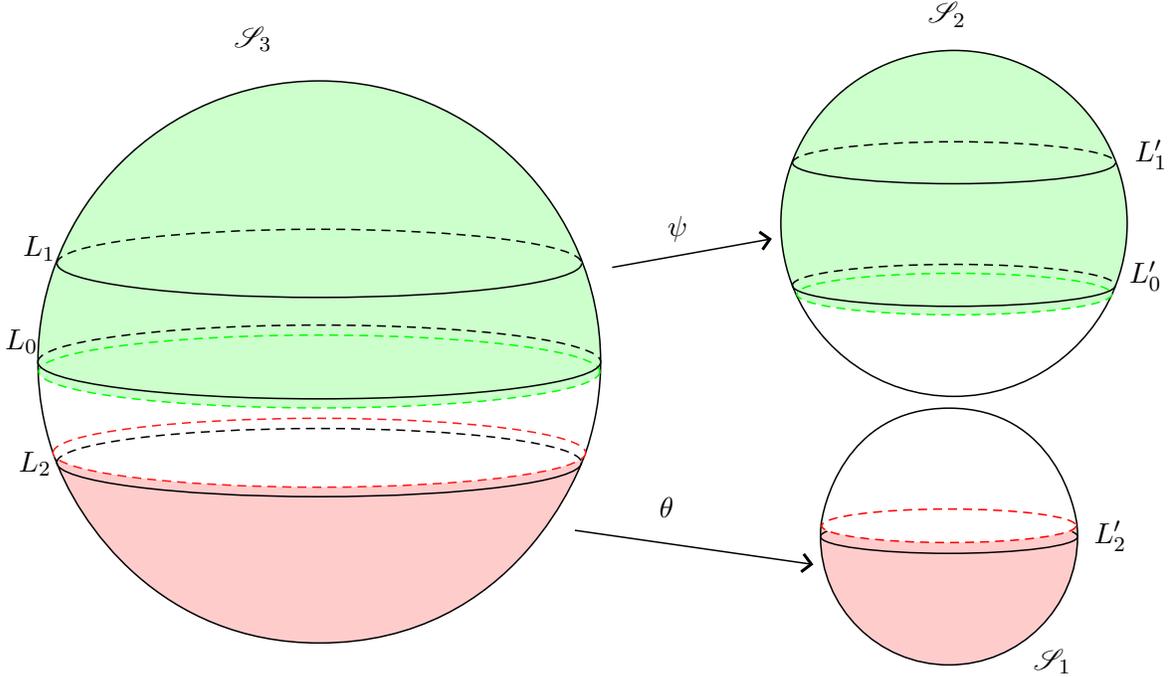
    \caption{The maps $\psi$ and $\theta$. They are respectively supported in the light green and light red shaded areas.}
    \label{fig:sphere_maps}
\end{figure}

\medskip

Now, let $H \colon [0,1]\times \mathscr{S}_3 \to \R$ be a Hamiltonian compactly supported in the northern hemisphere. We can assume up to shrinking the neighbourhood $\mathcal{N}(L_0)$ that $\mathcal{N}(L_0)$ misses the support of $H$. Note first that, by construction, $L_i$ is disjoint from $\Phi^1_H(L_j)$ whenever $i \neq j$. In order to have well-defined Floer complexes, we need to have transversal intersections. To this effect, for a small $\varepsilon>0$ we fix $F_\varepsilon$ a Hamiltonian compactly supported in $\mathcal{N}(L_0)$ whose $C^\infty$-norm is smaller than $\varepsilon$ and such that $\Phi_{F_\varepsilon}(L_0)$ intersects $L_0$ transversally in only two points $p_0$ and $q_0$. In order to fix the notations, we assume that $p_0$ represents the fundamental class in $HF_*(L_0, F_\varepsilon; \mathcal{N}(L_0))$, see Figure \ref{fig:sphere_links}. We denote by $p_2$ and $q_2$ the two transverse intersection points of $L_2$ and $\Phi_{f_*F_\varepsilon}^1(L_2)$ that are the respective images of $p_0$ and $q_0$ by $f$. We also denote by $p_0'$ and $q_0'$ the images of $p_0$ and $q_0$ by $\psi$ as well as $p_2'$ and $q_2'$ the images of $p_2$ and $q_2$ by $\theta$.

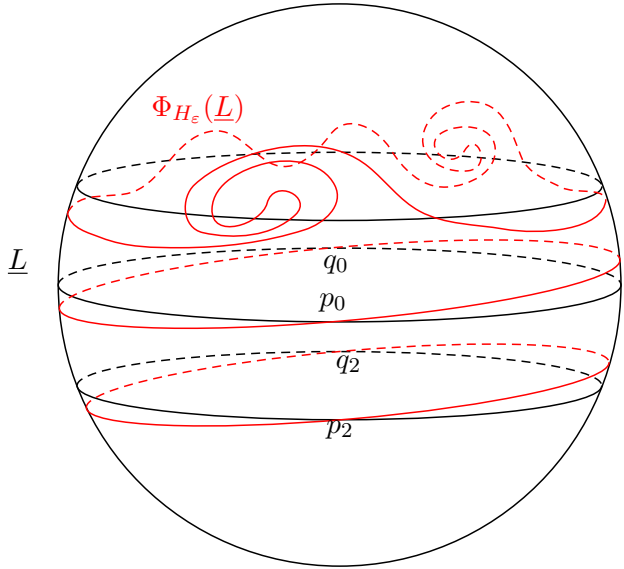
\begin{figure}[ht]
    \centering
\begingroup%
  \makeatletter%
  \providecommand\color[2][]{%
    \errmessage{(Inkscape) Color is used for the text in Inkscape, but the package 'color.sty' is not loaded}%
    \renewcommand\color[2][]{}%
  }%
  \providecommand\transparent[1]{%
    \errmessage{(Inkscape) Transparency is used (non-zero) for the text in Inkscape, but the package 'transparent.sty' is not loaded}%
    \renewcommand\transparent[1]{}%
  }%
  \providecommand\rotatebox[2]{#2}%
  \newcommand*\fsize{\dimexpr\f@size pt\relax}%
  \newcommand*\lineheight[1]{\fontsize{\fsize}{#1\fsize}\selectfont}%
  \ifx\svgwidth\undefined%
    \setlength{\unitlength}{280.46480194bp}%
    \ifx\svgscale\undefined%
      \relax%
    \else%
      \setlength{\unitlength}{\unitlength * \real{\svgscale}}%
    \fi%
  \else%
    \setlength{\unitlength}{\svgwidth}%
  \fi%
  \global\let\svgwidth\undefined%
  \global\let\svgscale\undefined%
  \makeatother%
  \begin{picture}(1,0.75802177)%
    \lineheight{1}%
    \setlength\tabcolsep{0pt}%
    \put(0,0){\includegraphics[width=\unitlength,page=1]{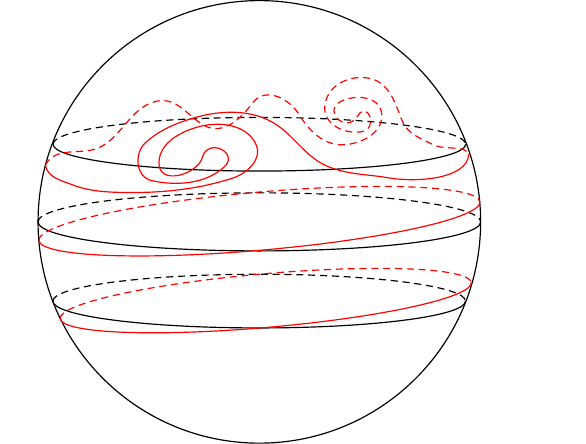}}%
    \put(0.42562444,0.17997909){\color[rgb]{0,0,0}\makebox(0,0)[lt]{\lineheight{1.25}\smash{\begin{tabular}[t]{l}$p_2$\end{tabular}}}}%
    \put(0.43791531,0.2679058){\color[rgb]{0,0,0}\makebox(0,0)[lt]{\lineheight{1.25}\smash{\begin{tabular}[t]{l}$q_2$\end{tabular}}}}%
    \put(0.41616999,0.35110529){\color[rgb]{0,0,0}\makebox(0,0)[lt]{\lineheight{1.25}\smash{\begin{tabular}[t]{l}$p_0$\end{tabular}}}}%
    \put(0.42089721,0.40215941){\color[rgb]{0,0,0}\makebox(0,0)[lt]{\lineheight{1.25}\smash{\begin{tabular}[t]{l}$q_0$\end{tabular}}}}%
    \put(-0.00266368,0.4002686){\color[rgb]{0,0,0}\makebox(0,0)[lt]{\lineheight{1.25}\smash{\begin{tabular}[t]{l}$\underline L$\end{tabular}}}}%
    \put(0.19020792,0.60637632){\color[rgb]{0,0,0}\makebox(0,0)[lt]{\lineheight{1.25}\smash{\begin{tabular}[t]{l}{\color{red}$\Phi_{H_\varepsilon}(\underline L)$}\end{tabular}}}}%
  \end{picture}%
\endgroup%

    \caption{The link $\underline L$ and its image by $\Phi_{H_\varepsilon}$}
    \label{fig:sphere_links}
\end{figure}

\noindent We can assume, up to a $C^\infty$-small perturbation of $H$, that $\Phi_H^1(L_1)$ intersects $L_1$ transversally. Then, we set
\[H_\varepsilon\coloneqq H+F_\varepsilon+ f_*F_\varepsilon.\]
With these definitions, we have that $\Phi_{H_\varepsilon}(L_i)$ and $L_j$ intersect transversally for all $i,j$ with empty intersection if $i\neq j$. Furthermore, our perturbation of $H$ is $C^\infty$-small, that is without any consequence on our statements on Hofer-Lipschitz quasimorphisms. We define additionally
\[H_\varepsilon''\coloneqq(H+F_\varepsilon)\circ \psi^{-1}, \quad H_\varepsilon'\coloneqq F_\varepsilon \circ \theta^{-1}\]
two time-dependent Hamiltonians on $\mathscr{S}_2$ and $\mathscr{S}_1$, which are extended by 0 away from the image of $\psi$ and $\theta$ respectively.

\medskip

Since the link $L_0' \cup L_1'$ in $\mathscr{S}_2$ has monotonicity constant $\lambda$ and as well as the link $L_2'$ in $\mathscr{S}_2$, we can define spectral invariants $c_{L_0' \cup L_1'}$ and $c_{L_2'}$ using the construction described in Section \ref{subsec.def of spec inv}. Proposition \ref{prop.reducing to comp supp Ham} follows from the next two results.

\begin{prop}\label{prop.hard ineq of spec inv}
Let $H$ be a Hamiltonian compactly supported in the northern hemisphere of $\mathscr{S}_3$. Let $H_\varepsilon$, $H_\varepsilon''$, and $H_\varepsilon'$ be the three Hamiltonians described above. Then,
\[3 c_3(H_\varepsilon) \leq 2 c_{L_0' \cup L_1'}(H_\varepsilon'')+ c_{L_2'}(H_\varepsilon').\]
\end{prop}

\begin{prop}\label{prop.eq of spec inv on diff spheres}
Let $H$ be a Hamiltonian compactly supported in the northern hemisphere of $\mathscr{S}_3$, and let $H_\varepsilon$ and $H_\varepsilon''$ be the two Hamiltonians as above. Then,
\[c_2(H_\varepsilon)=c_{L_0' \cup L_1'}(H_\varepsilon'').\]
\end{prop}

\begin{rk}
In theory, there is no need to perturb $H$ in order to state Propositions \ref{prop.hard ineq of spec inv} and \ref{prop.eq of spec inv on diff spheres}, we could also have written
\[3c_3(H) \leq 2 c_{L_0' \cup L_1'}(H\circ \psi^{-1}) \, \textrm{ and }\, c_2(H)=c_{L_0' \cup L_1'}(H \circ \psi^{-1}).\]
\end{rk}

Combining the results of the above propositions we have a proof of Proposition \ref{prop.reducing to comp supp Ham}.

\begin{proof}[Proof of Proposition \ref{prop.reducing to comp supp Ham}]
We have
\begin{align*}
    \left( 3c_3 - 2c_2 -c_1\right)(H_\varepsilon) & \leq 2 c_{L_0' \cup L_1'}(H_\varepsilon'')+ c_{L_2'}(H_\varepsilon')-2c_2(H_\varepsilon) -c_1(H_\varepsilon)\\
    & = c_{L_2'}(H_\varepsilon')-c_1(H_\varepsilon),
\end{align*}
where the first inequality comes directly from Proposition \ref{prop.hard ineq of spec inv} and the second from Proposition \ref{prop.eq of spec inv on diff spheres}. Since $H_\varepsilon$ is a $C^\infty$-small perturbation of $H$ both summand on the right-hand side go to 0 as $\varepsilon$ goes to 0, it follows that $(3c_3-2c_2-c_1)(H) \leq 0$. This inequality holds for any Hamiltonian supported inside the interior of the northern hemisphere, in particular for its iterates. This implies directly that $r_\lambda(\Phi_H)\leq 0$ and since $\overline{H}$ is also compactly supported in the northern hemisphere, we have as well
\[r_\lambda(\Phi_H)=-r_\lambda(\Phi_H^{-1})=-r_\lambda(\Phi_{\overline{H}}) \geq 0.\]
Therefore, $r_\lambda(\Phi^1_H)=0$ as we wanted to show.
\end{proof}

\subsection{Proof of Proposition \ref{prop.eq of spec inv on diff spheres}}\label{subsec.proof of prop 4.4}

In this section, we prove Proposition \ref{prop.eq of spec inv on diff spheres}. In the proof, we exhibit an isomorphism at the chain level between two Floer complexes for some choice of almost complex structure. As before $H$ denotes a Hamiltonian on $\mathscr{S}_3$ compactly supported in the northern hemisphere, $H_\varepsilon$ denotes a suitable perturbation of $H$ and $H_\varepsilon''$ denotes the transferred Hamiltonians on $\mathscr{S}_2$.

\begin{prop}\label{prop.iso of chain complexes S3 and S2}
There is a natural isomorphism of filtered $\Lambda$-modules
\[CF^t_*(L_1 \cup L_2, H_\varepsilon) \cong CF_*^t(L_0' \cup L_1', H_\varepsilon'')\]
that yields a quasi-isomorphism of chain complexes, i.e., an isomorphism at the filtered homology level.
\end{prop}

\begin{proof}
We denote $\omega_3$ and $\omega_2$ the two symplectic forms on $\mathscr{S}_3$ and $\mathscr{S}_2$ and by $\omega_{3,X}$ and $\omega_{2,X}$ the respective Perutz-type symplectic forms on the respective symmetric products. The proof is in two parts; in the first part we define the isomorphism of filtered $\Lambda$-module while in the second part we show that for some almost complex structure the two complexes are identified. A priori, the almost complex structure is not both $\omega_3$- and $\omega_2$-tame. However, as long as the almost complex structure is picked close enough to the symmetrisation of an almost complex structure on the sphere, we will have tameness in both cases.

Let $\psi'\colon \mathscr{S}_3 \to \mathscr{S}_2$ be a diffeomorphism that satisfies the following two conditions:
\begin{itemize}
\setlength\itemsep{-1pt}
    \item[$\bullet$] $\psi'$ and $\psi$ coincide on a disc $D_1$ that contains the support of $H$ and avoids $\mathcal{N}(L_0)$.
    \item[$\bullet$] $\psi' \rvert_{\mathcal{N}(L_2)}=\psi\lvert_{\mathcal{N}(L_0)}\circ f^{-1}$. 
\end{itemize}
Note that $\psi$ is \textit{not} a symplectomorphism since $\mathscr{S}_3$ and $\mathscr{S}_2$ do not have the same total areas. Due to our definition of $H_\varepsilon$, we have $H_\varepsilon''=H_\varepsilon\circ \psi'^{-1}$. This implies in particular that $\psi'(\Phi_{H_\varepsilon}(L_i))=\Phi_{H_\varepsilon}(L_{2-i}')$ for $i=1,2$.

\noindent The diffeomorphism $\psi'$ naturally induces a map of the symmetric product.
\begin{align*}
\Psi' \coloneqq \Sym^2(\psi') \colon \Sym^2(\mathscr{S}_3) &\to \Sym^2(\mathscr{S}_2)\\
[x,y] &\mapsto [\psi'(x),\psi'(y)]
\end{align*}
The diffeomorphism $\Psi'$ sends $\Sym^2(L_1 \cup L_2)$ to $\Sym^2(L_1' \cup L_0')$ as well as their respective images under $\Phi_{H_\varepsilon}$ and $\Phi_{H_\varepsilon''}$. In particular, it induces a bijection on the Lagrangian intersections. We want to show that it actually induces a bijection on the generators of the Floer complexes, that is to say that the bijection on the intersection points can be lifted as a bijection on the capped intersection (up to the equivalence relation of (\ref{eq.equiv of cappings})). This follows from the fact that index 2 Maslov discs in the two groups
\[\pi_2(\Sym^2(\mathscr{S}_3), L_1 \cup L_2), \quad  \pi_2(\Sym^2(\mathscr{S}_2), L_1' \cup L_0')\]
have the same ``bulk-deformed" area in the sense of equation (\ref{eq.equiv of cappings}). We give now a bit more details.

\medskip

We fix some point
\[x=[x_1,x_2]\in \Sym^2(\Phi_{H_\varepsilon}(L_1 \cup L_2)) \pitchfork \Sym^2(L_1 \cup L_2),\]
where $x_i$ lies in $L_i$ for $i=1,2$, as a reference point for $CF_*(L_1 \cup L_2, H_\varepsilon)$ and its image $\Psi'(x)$ as a reference point for $CF_*(L_0' \cup L_1', H_\varepsilon'')$. Then, the reference path $\gamma_x:t\mapsto \Sym \Phi^t_{H_\epsilon}(x)$ is the product of a constant path $\gamma_2$ at $x_2$, and a path $\gamma_1$ starting at $x_1$ and staying inside $D_1$. For some other Lagrangian intersection point $y=[y_1,y_2] \in \Sym^2(\mathscr{S}_3)$ with capping $\widehat{y}$, we can find another capping $\widehat{y_0}$ consisting of a product of a capping from $\gamma_1$ to $y_1$ staying inside $D_1$ and a capping from $\gamma_2$ to $y_2$ contained in $\mathcal N(L_2)$. In particular, $\widehat{y_0}$ does not intersect the diagonal of $\Sym^2(\mathscr{S}_3)$. Then, according to Remark \ref{rk:recapping_with_discs} we can decompose $[\widehat{y}]$ as the class of capping $[\widehat{z} \# \widehat{y_0}]$, where
\[\widehat{z} \in \pi_2(\Sym^2(\mathscr{S}_3), L_1 \cup L_2).\]
Then if we denote $\widehat{y}'=\widehat{z}' \# \widehat{y_0}'$ the corresponding capping in $\Sym^2(\mathscr{S}_2)$ and $\eta''$ the real number such that the link $L'_1 \cup L'_0$ is $\eta''$-monotone, we have
\begin{align*}\mathcal A_{H_\epsilon''}(\Psi'_*[\widehat{y}]) &=\int_0^1\Sym^2(H_{\epsilon}'')(\Psi'(x))dt-\int_{[0,1]\times[0,1]} \left(\widehat{y}'\right)^*\omega_{2,X} - \eta ''\widehat{y}'\cdot \Delta\\
&=\int_0^1 \Sym^2(H_\varepsilon)(x) dt- \int_{[0,1]\times[0,1]} \left(\widehat{y_0}'\right)^*\omega_{2,X}\\
&-\int_{[0,1]\times[0,1]} \left(\widehat{z}'\right)^*\omega_{2,X}-\eta ''\widehat{z}'\cdot \Delta\\
&=\int_0^1 \Sym^2(H_\varepsilon)(x) dt- \int_{[0,1]\times[0,1]} \left(\widehat{y_0}\right)^*\omega_{3,X}\\
&-\int_{[0,1]\times[0,1]} \left(\widehat{z}\right)^*\omega_{3,X}-\eta \,\widehat{z}\cdot \Delta\\
&=\mathcal A_{H_\epsilon}([\widehat y])
\end{align*}
The first equality is just the definition of the action for the $\eta''$-monotone link $L'_0 \cup L_1'$ while the second equality just comes from expanding the capping $\widehat{y}'$ and the fact that $\widehat{y_0}'\cdot \Delta=0$. The third equality is a bit trickier; it comes from the fact that on one hand the map $\Psi'$ is a symplectomorphism when restricted to
\[\Sym^1(D_1)\times \Sym^1(\mathcal{N}(L_2))\subset \Sym^2(\mathscr{S}_3)\]
and that on the other hand the bulk parameter is chosen to make the two Lagrangians monotone with the same monotonicity constant $\lambda$, and therefore Proposition \ref{prop:monotonicity} implies:
\[\int_{[0,1]\times[0,1]} \left(\widehat{z}'\right)^*\omega_{2,X}+\eta ''\widehat{z}'\cdot \Delta=\frac \lambda 2 \mu\left(\widehat{z}'\right)=\frac \lambda 2 \mu\left(\widehat{z}\right)=\int_{[0,1]\times[0,1]} \left(\widehat{z}\right)^*\omega_{3,X}+\eta \,\widehat{z}\cdot \Delta.\]
This shows that $\Psi'$ induces an isomorphism of filtered modules
\[\Psi'_* \colon CF^t_*(L_1 \cup L_2, H_\varepsilon) \xrightarrow{\cong} CF_*^t(L_0' \cup L_1', H_\varepsilon'')\]
It remains to show that this isomorphism is an isomorphism at the homology level. In practice, as announced, we show something stronger. We exhibit a family of almost complex structures $(J_s)_{s \in [0,1]}$ such that the complex 
\[CF^t_*(L_1 \cup L_2, H_\varepsilon, J_s)\]
is isomorphic to the complex
\[CF^t_*(L'_1 \cup L'_2, H''_\varepsilon, \Psi'_*J_s).\]
To this effect, we pick an almost complex structure $j$ on $\mathscr{S}_3$. Note that $j$ is $\omega_3$-tame and that $j' \coloneqq \Psi'_*j$ is $\omega_2$-tame. As a consequence, the almost complex structure $\Sym^2(j)$ is $\omega_{3,X}$-tame while $\Sym^2(j')$ is $\omega_{2,X}$-tame. Let \smash{$(J_s)_{s\in[0,1]}$} be a generic small perturbation of $\Sym^2(j)$ that is $(V,\omega_3,j)$-nearly symmetric for some neighbourhood $V$ of the diagonal in $\Sym^2(\mathscr{S}_3)$. Then $(\Psi_*'J_t)_t$ will be $(\Psi'_*V, \omega_2, j')$-nearly symmetric since tameness is preserved under small perturbations. Then $\Psi'$ provides a one-to-one correspondence between the $J_t$-holomorphic curves and the $\Psi_*'J_s$-holomorphic curves, which completes the proof of the proposition.
\end{proof}

The same goes well for the continuation maps (as long as the isotopy of the Hamiltonians is supported in the northern hemisphere).

\begin{prop}\label{prop.commutativity with cont maps}
Let $H^s$ be a homotopy between two Hamiltonians $H$ and $G$ of the sphere $\mathscr{S}_3$ such that $H^s$, $H$ and $G$ are all compactly supported in the northern hemisphere $D_N$. As in Section \ref{subsec.setting and notations} we can define the Hamiltonians $H_\varepsilon$, $H_\varepsilon''$, $G_\varepsilon$, $G_\varepsilon''$ and $H^s_\varepsilon$. Then the isomorphisms of Proposition \ref{prop.iso of chain complexes S3 and S2} commute with the continuation maps induced by $H^s_\varepsilon$ and $H_\varepsilon^s \circ \psi^{-1}$. That is, we get a commutative diagram
\[
\begin{tikzcd}
HF_*(L_1 \cup L_2, H_\epsilon) \arrow{r}{\cong} \arrow[swap]{d} & HF_*(L_0'\cup L_1', H_\epsilon'') \arrow{d}\\
HF_*(L_1 \cup L_2, G_\varepsilon)  \arrow{r}{\cong} & HF_*(L_0'\cup L_1', G_\varepsilon''),
\end{tikzcd}    
\]
where the vertical arrows are induced by the continuation maps.
\end{prop}

\begin{proof}
The proof of Proposition \ref{prop.iso of chain complexes S3 and S2} is easily adaptable in order to count the pseudo-holomorphic curves arising in the continuation map induced by $H^s$ and $H^s \circ \psi^{-1}$.
\end{proof}

The proposition \ref{prop.eq of spec inv on diff spheres} then follows from the above propositions.

\begin{proof}[Proof of Proposition \ref{prop.eq of spec inv on diff spheres}]
We use Proposition \ref{prop.iso of chain complexes S3 and S2} with $G$ some $C^\infty$-small Hamiltonian compactly supported in the northern hemisphere, so the homologies $HF_*(L_1 \cup L_2, G_\varepsilon)$ and $HF_*(L_1' \cup L_0', G_\varepsilon'')$ are canonically identified with the relative quantum homology of the corresponding Lagrangians. From Proposition \ref{prop.iso of chain complexes S3 and S2}, the top arrow in the commutative diagram of Proposition \ref{prop.commutativity with cont maps} respects filtration. This implies that the spectral invariants on both complexes are equal,
\[c_2(H_\varepsilon)=c_{L_0' \cup L_1'}(H_\varepsilon'').\]
\end{proof}

\subsection{An almost Künneth formula for link Floer homology}\label{subsec.almost Künneth formula}

In this section, our aim is to prove the proposition \ref{prop.hard ineq of spec inv}. Our proof relies on having some control on the holomorphic curves appearing in the Floer complexes and in the continuation maps, the proof of these more technical results will be postponed to Section \ref{sec.neck-stretching and constrains}. 

\medskip

The first result of this section is to describe the isomorphism we would hope for, forgetting about the differential. One can see that this is very similar to the first statement of Proposition \ref{prop.iso of chain complexes S3 and S2}. First, we recall the definition of the tensor product of two filtered chain complexes.

\begin{dfn}\label{def.iso of chain cplx}
Let $(C_*^1, \partial^1, \nu^1)$ and $(C_*^2, \partial^2, \nu^2)$ be two filtered chain complexes where $\nu_i \colon C^i_* \to \R\cup\{-\infty\}$ is the filtration. Then, the product of two filtered chain complexes is given by $(C_*^1 \otimes C_*^2, \partial, \nu \coloneqq \nu^1+\nu^2)$, where the standard product differential $\partial$ is defined by
\[\partial (c^1 \otimes c^2)=c^1 \otimes \partial c^2+ \partial c^1 \otimes c^2\]
and the filtration $\nu$ is defined by
\[\nu(c^1 \otimes c^2)=\nu^1(c^1)+\nu^2(c^2).\]
\end{dfn}

\noindent We point out that at the homology level, the Künneth formula for filtered homology is more complicated that its unfiltered counterpart, this is due to the fact that finite bars do not have trivial projective resolutions, see \cite[Section 2.2]{PSS} for details. We can now state the next proposition.

\begin{prop}\label{prop.iso of filtered module S1 S2 and S2}
    There is a natural isomorphism of filtered $\Lambda$-modules 
    \[g:CF_*(L_0'\cup L_1', H_\epsilon'')\otimes_\Lambda CF_*(L_2',H_\epsilon')\xrightarrow{\cong} CF_*(L_0 \cup L_1 \cup L_2, H_\epsilon)\,.\]
\end{prop}

\begin{proof} 
We point out to the reader that the Novikov formal variable $T$ acts on the three Floer complexes in the proposition by decreasing the action by $\lambda$. This is obviously a requirement if we want to have a filtered isomorphism. The proof is very similar in spirit to the proof of the first part of Proposition \ref{prop.iso of chain complexes S3 and S2}.

The maps $\psi \colon D_N \cup \mathcal{N}(L_0) \to \mathscr{S}_2 $ and $\theta  \colon D_2 \cup \mathcal{N}(L_2) \to \mathscr{S}_1$ defined in Section \ref{subsec.setting and notations} give rise to a bijection between the symmetrisations of the links as follows
\begin{align*}
\Psi \times \Theta \colon & \Sym^3(L_0 \cup L_1 \cup L_2) \to \Sym^2(L'_0 \cup L'_1) \times \Sym^1(L'_2)\\
&[x_0,x_1,x_2] \to [\psi(x_0),\psi(x_1)] \times [\theta(x_2)]=:[x_0',x_1']\times [x_2'],
\end{align*}
where $x_0$, $x_1$ and $x_2$ belongs to $L_0$, $L_1$ and $L_2$ respectively. Since the link component $L_i$ is disjoint from $\Phi^1_H(L_j)$ whenever $i \neq j$, the map also induces a bijection on the transverse intersections. We need to take care of the cappings.

We fix some point
\[x=[x_0,x_1,x_2]\in \Sym^3(\Phi_{H_\varepsilon}(L_0 \cup L_1 \cup L_2)) \pitchfork \Sym^3(L_0 \cup L_1 \cup L_2),\]
as a reference path for $CF_*(L_0 \cup L_1 \cup L_2,H_\varepsilon)$. For some other Lagrangian intersection point $y=[y_0,y_1,y_2] \in \Sym^3(\mathscr{S}_3)$, with capping $\widehat{y}$. According to Remark \ref{rk:recapping_with_discs}, $\widehat{y}$ is in the same class as a capping $\widehat{z} \# \widehat{s}$, where
\[\widehat{z} \in \pi_2(\Sym^3(\mathscr{S}_3), L_0 \cup L_1 \cup L_2)\]
and $\widehat{s}$ is a capping that is the product of a capping $\widehat{s_1}$ in $D_1 \cup \mathcal{N}(L_0)$ (where $D_1$ is some disc disjoint from $\mathcal{N}(L_0)$ that contains the support of $H$) and a capping $\widehat{s_2}$ in $\mathcal{N}(L_2)$ such that $\widehat{s}$ does not intersect the diagonal in $\Sym^3(\mathscr{S}_3)$. Then, $\psi$ and $\theta$ induce a natural map on the generators of the form $([y_0,y_1,y_2], [\widehat{s}])$
\begin{align*}
    \Psi \times \Theta \colon ([y_0,y_1,y_2], [\widehat{s}]) \mapsto ([y_0',y_1'],[\widehat{s_1}']) \times ([y_2'],[\widehat{s_2}']),
\end{align*}
where $\widehat{s_1}' \coloneqq \widehat{s_1} \circ \Sym^2(\psi)^{-1}$ and $\widehat{s_2}' \coloneqq \widehat{s_2} \circ \theta^{-1}$. Since $\widehat s$ and $\widehat{s_1}'$ are away from the diagonal in $\Sym^3(\mathscr{S}_3)$ and $\Sym^2(\mathscr{S}_2)$ respectively, and since $\psi$ is area preserving, this construction preserves the action. For the other term of the capping $\widehat{z}$ we see that adding it to the capping $\widehat{s}$ amounts to multiplying by some formal variable $T^m \in \Lambda$. We can then extend the map $\Psi \times \Theta$ to all the generators of the Floer complex this way:
\begin{align*}
    \Psi \otimes \Theta \colon  & CF_*(L_0 \cup L_1 \cup L_2, H_\varepsilon) \to CF_*(L_0'\cup L_1', H_\varepsilon'')\otimes_\Lambda CF_*(L_2',H_\varepsilon')\\
    &([y_0,y_1,y_2], [\widehat{y}]) \mapsto T^k([y_0',y_1'],[\widehat{s_1}']) \otimes_{\Lambda} T^\ell([y_2'],[\widehat{s_2}']),
\end{align*}
for any $k$ and $\ell$ in $\Z$ such that $k+\ell=m$. Because we have tensored over the Novikov field $\Lambda$ any such choice give the same result in the tensor product. This maps preserves the action since the formal variable $T$ acts by decreasing the action by $\lambda$ in all complexes and it is also clearly bijective.
\end{proof}

In contrast to Proposition \ref{prop.commutativity with cont maps}, we did not manage to prove that the above isomorphism is a chain morphism or that it commutes with the continuation maps. We conjecture that this is true, but we will not need it for our purposes.

\medskip

Let $G$ be a time-dependent Hamiltonian supported in the northern hemisphere that is $C^\infty$-small. We also consider, as in Proposition \ref{prop.commutativity with cont maps}, an isotopy $(H^s)_s$ between $H$ and $G$, this defines the isotopies 
\[H_\varepsilon^s\coloneqq H^s+F_\varepsilon+F_\varepsilon\circ f\]
and $H_\varepsilon^s\circ \psi^{-1}$ that we will denote just by $H^s$ and $H''^s$ in order to have simpler notations. For some choice of two homotopies of isotopies of almost complex structure for now both denoted $J^s$ on $\mathscr{S}_3$ and $\mathscr{S}_2$, we get continuation maps
\[\Psi_{(H^s,J^s)} \colon CF_*(L_0 \cup L_1 \cup L_2, H_\varepsilon, J^0) \to CF_*(L_0 \cup L_1 \cup L_2, G_\varepsilon, J^1)\]
and
\[\Psi_{(H''^s,J^s)} \colon CF_*(L_0' \cup L_1', H''_\varepsilon, J^0) \to CF_*(L_0' \cup L_1', G''_\varepsilon, J^1).\]

The complexes $CF_*(L_2', H_\varepsilon', J)$ and $CF_*(L_2', G_\varepsilon', J)$ are canonically identified since $H_\varepsilon'=G_\varepsilon'$ on $\mathscr{S}_1$. All these maps along with the isomorphisms of filtered $\Lambda$-modules from Proposition \ref{prop.iso of filtered module S1 S2 and S2} yield the following diagram of filtered $\Lambda$-modules,
\begin{equation}\label{eq.almost commutative diag}
\begin{tikzcd}
CF_*(L_0'\cup L_1', H_\epsilon'',J^0)\otimes_\Lambda CF_*(L_2',H_\epsilon',J) \arrow{r}{g} \arrow[swap]{d}{\Psi_{(H''^s,J^s)}\otimes \,Id} & CF_*(L_0 \cup L_1 \cup L_2, H_\epsilon, J^0) \arrow{d}{\Psi_{(H^s,J^s)}}\\
CF_*(L_0'\cup L_1', G_\epsilon'',J^1)\otimes_\Lambda CF_*(L_2',G_\epsilon',J)  \arrow{r}{g} & CF_*(L_0 \cup L_1 \cup L_2, G_\epsilon,J^1).
\end{tikzcd}
\end{equation}
Note that $g$ is a priori not a chain morphism, and that this diagram is a priori \textit{not commutative nor quasi-commutative}.

Since $G_\varepsilon$ is $C^\infty$-small, one can canonically identify the homologies of the chain complexes of the bottom line of the above diagram with the quantum homology, since $\partial=0$ in this case \cite[Lemma 6.10]{CGHMSS22}. Moreover,
\begin{align*}
    QH_*(L_0 \cup L_1 \cup L_2, \Lambda) &\cong H_*(\mathbb{T}^3, \Lambda)\\ 
    &\cong H_*(\mathbb{T}^2, \Lambda) \otimes_\Lambda H_*(\mathbb{T},\Lambda)\\
    &\cong QH_*(L_0' \cup L_1',\Lambda) \otimes_{\Lambda} QH_*(L_2',\Lambda).
\end{align*}

We recall that in Section \ref{subsec.setting and notations}, we called $p_2'$ the intersection points of $L_2'$ and $\Phi_{G_\varepsilon'}(L_2')$ that corresponds to the fundamental class of $HF_*(L_2', F_\varepsilon\circ \theta^{-1}; \mathcal{N}(L_2'))$. Then $p_2'$ with suitable capping also corresponds to the fundamental class of $QH_*(L_2')$ under the isomorphism with $HF_*(L_2', G_\varepsilon', J)$. We prove the following key proposition on the representatives of the fundamental classes of Link Floer homology.

\begin{prop}[Almost Künneth formula]\label{prop.maps of fund class is fund class}
Let $H$ and $G$ be two Hamiltonians as before and $x$ be some chain element of $CF_*(L_0' \cup L_1', H_\varepsilon'')$ that represents the fundamental class. The chain element $p_2'$ as above in $CF_*(L_2', H_\varepsilon')$ represents the fundamental class too. As a consequence, in the tensor product of the chain complexes, $x \otimes p_2'$ is also a representative of the fundamental class. Let $g$ be the filtered isomorphism of Proposition \ref{prop.iso of filtered module S1 S2 and S2}. Then, there exists a choice of isotopies of almost complex structures such that 
\[g(x \otimes p_2') \in CF_*(L_0 \cup L_1 \cup L_2, H_\varepsilon)\]
is a cycle and its homology class is mapped to the fundamental class of 
\[HF_*(L_0\cup L_1 \cup L_2, G_\varepsilon)\]
by the continuation map $\Psi_{(H^s,J^s)}$.
\end{prop}

Proposition \ref{prop.maps of fund class is fund class} does not imply the quasi-commutativity of (\ref{eq.almost commutative diag}) but it gives enough information to prove Proposition \ref{prop.hard ineq of spec inv}.

\begin{proof}[Proof of Proposition \ref{prop.hard ineq of spec inv}]
Let $x$ be a representative of the fundamental class of minimal action in $CF_*(L_0' \cup L_1',H_\varepsilon'')$, by definition \smash{$\mathcal{A}(x)=2c_{L_0' \cup L_1'}(H_\varepsilon'')$}.

From Proposition \ref{prop.maps of fund class is fund class}, we get that $g(x \otimes p_2')$ is also a cycle and its homology class is the fundamental class. As a consequence,
\begin{align*}
    3c_3(H_\varepsilon)&\coloneqq 3c_{L_0 \cup L_1 \cup L_2}(H_\varepsilon)\\
    &\leq \mathcal{A}(g(x \otimes p_2'))\\
    &= \mathcal{A}(x \otimes p_2')\\
    &=\mathcal{A}(x)+\mathcal{A}(p_2')\\
    &=2c_{L_0' \cup L_1'}(H_\varepsilon'')+c_{L_2'}(H_\varepsilon').
\end{align*}
The first inequality comes from the fact $[g(x \otimes p_2')]$ is the fundamental class, the second comes from the fact that $g$ is a filtered isomorphism and the rest follow from the definition of the filtration on a tensor of chain complexes as well as the definition of the spectral invariants. This proves Proposition \ref{prop.hard ineq of spec inv}.
\end{proof}

We are left with Proposition \ref{prop.maps of fund class is fund class}, we do this in the next section. The proof relies on understanding the holomorphic curves counted in diagram (\ref{eq.almost commutative diag}). This is possible for some curves for some specific choice of ``stretched" almost complex structure.

\section{Neck-stretching and constraints on Floer strips}\label{sec.neck-stretching and constrains}

In this section, we will give a proof of Proposition \ref{prop.maps of fund class is fund class}. The proof will be in two steps. We will exhibit almost complex structures, by using two neck-stretching arguments, that verify some helpful properties. First, in Lemma \ref{lem.stretch L0' L1'}, we will stretch the neck of some cylinder between $L_0'$ and $L_1'$ while in the proof of Lemma \ref{lem.stretch L0 L2}, we will stretch along some cylinder between $L_0$ and $L_2$. The end result will be a family of almost complex structures on the symmetric product of $\mathscr{S}_3$ that has simple holomorphic curves with starting points of the form $[p_2,\cdot,\cdot]$. By transferring accordingly this family onto the symmetric products of $\mathscr{S}_2$ and $\mathscr{S}_1$ this will yield the ``almost" Künneth formula of Section \ref{subsec.almost Künneth formula}. We first review some known results for transversality of pseudo-holomorphic curves.

\subsection{Transversality results for Link Floer Homology}\label{subsec.some trans results}
We will need to use the following result on transversality for pseudo-holomorphic curves. Let $M$ be a symplectic manifold and $\mathcal U \subset M$ some open set of $M$. We denote $\mathcal J(M)$ the space of isotopies $\{J_t\}_{t \in [0,1]}$ of $\omega$-tamed almost complex structure on $M$. We know that this space is non-empty and contractible. For $\textbf{J} \in \mathcal J(M)$, we denote
\[\mathcal J_\textbf{J}(\mathcal U) \coloneqq \left\{J \in \mathcal J(M), J\lvert_{\mathcal U^c} \equiv \textbf{J} \right\}.\]
Mimicking Proposition 5.13 of \cite{KhSe} that uses the results of \cite{FHS}, we get the following proposition.

\begin{prop}\label{prop.general transversality}
Let $L_1$ and $L_2$ be closed Lagrangians in $M$. Fix $\textbf{J} \in \mathcal J(M)$ and an open set $\mathcal U \subset M$. Then, for a $C^\infty$-generic $J \in \mathcal J_\textbf{J}(\mathcal U)$, all curves in the Floer complex $CF_\ast(L_0,L_1)$ that are not entirely contained in $\mathcal U^c$ are transversely cut-out.
\end{prop}

\begin{proof}[Sketch of Proof]
Let $J_t \in \mathcal J_\textbf{J}(\mathcal U)$ and let $u \colon \mathbb R_s \times [0,1]_t \to M$ be a $J_t$-holomorphic curve of the Floer complex. Then, according to \cite{FHS}, the space of regular points of $u$,
\[C(u) \coloneqq \left \{(s,t) \in \R \times (0,1)\mid  du(s,t) \neq 0, ~~ u(s,t) \not\in u(\R \setminus \{s\},t) \right\}\]
is open and dense in $\R \times (0,1)$. Since $\mathcal U$ is open. If $\text{Im}(u) \not\subset \mathcal U^c$, then $u^{-1}(\mathcal U)$ is a non-empty open set of $\R \times (0,1)$. In particular, $u^{-1}(\mathcal U) \cap C(u)$ is open and dense in $u^{-1}(\mathcal U)$. Then one can run the standard proof of transversality for pseudo-holomorphic curves. For $(s,t) \in u^{-1}(\mathcal U) \cap C(u)$ one can modify $J_t$ in a small box around $u(s,t)$ to obtain transversality.
\end{proof}

This has the consequence that one can further assume that the neighbourhood $V$ in Definition \ref{dfn.nearly symmetric acs} contains complex divisors of the form
\[\{z_i\}_{i=1}^m \times \Sym^{k-1}(\Sigma)\]
for points \(\{z_i\}_{i=1}^m \in \Sigma\) in the complement of both the link $\underline{L}$ and the link $\Phi_H(\underline L)$ and still have transversality for $C^\infty$-generic $(V,\omega,j)$-nearly symmetric almost complex structures. This means that on a neighbourhood of the divisors the almost complex structure is symmetric (hence complex). This appears in \cite{OS}.

\medskip

\noindent Indeed, let
\[N_i \coloneqq D_i\times \Sym^{k-1}(\Sigma)= \mathcal N (\{z_i\}\times \Sym^{k-1}(\Sigma))\] 
be a neighbourhood of the complex divisor with $D_i \subset \Sigma$ a small disk around $z_i$. If the almost complex structure is the restriction of a symmetric almost complex structure $\Sym^k(j)$ on $N_i$ then any pseudo-holomorphic strip $u$ that intersects $N_i$ is either completely contained inside a set of the form $\{z\} \times \Sym^{k-1}(\Sigma)$ with $z \in D_i$ or leaves the neighbourhood $N_i$. Let us show this property quickly, since some variation will appear later. We denote $\pi_1 \colon N_i \to D_i$ the projection on the first coordinate, it is a holomorphic map. Then upon appropriate restriction of the domain of definition, the map $ \pi_1 \circ u$ is well defined and holomorphic from an open subset of $\R \times (0,1)$ to $D_i$. From the maximal principle, this map is either constant or eventually leaves $D_i$.

In order to prove the transversality for $(V,\omega,j)$-nearly symmetric structure we have to show that the pseudo-holomorphic strips that we consider are never as in the first case but since $u$ is asymptotic to an intersection point of $\underline{L}$ and $\Phi_H(\underline L)$ that is not contained in $N_i$ this case can indeed never happen.

\medskip

The same argument also works for continuation maps as long as the point $z \in \Sigma$ is away from the trace of the isotopy $\Sym(\Phi^t_H(\underline L))$.

\subsection{Proof of Proposition \ref{prop.maps of fund class is fund class}}

We will use the notations of Section \ref{subsec.setting and notations} for $H$, $H_\varepsilon$, $H_\varepsilon'$ and $H_\varepsilon''$ as well as the corresponding Hamiltonians for $G$ and also for an isotopy $H^s$ between $H$ and $G$ as defined in Section \ref{subsec.almost Künneth formula}. We distinguish the south pole of $\mathscr{S}_2$ denoted by $\tau_2$. Likewise, we denote by $\tau_1$ the north pole of $\mathscr{S}_1$. Appropriate stretching of the neck in $\mathscr{S}_3$ would amount to wedge the two pointed spaces $(\mathscr{S}_1,\tau_1)$ and $(\mathscr{S}_2,\tau_2)$. We first prove the following preliminary lemma.

\begin{lemma}\label{lem.stretch L0' L1'}
There exists an almost complex structure $j$ on $(\mathscr{S}_2,\omega_2)$ and a neighbourhood $V$ of 
\[\Delta \cup \{\tau_2\} \times \Sym(\mathscr S_1),\]
such that some isotopy $(J^s)_{s\in [0,1]}\coloneqq (J^{s,t})_{s,t\in [0,1]}$ of families of $(V,\omega_2,j)$-nearly symmetric almost complex structures verifies the following properties.
\begin{enumerate}[label=(\roman*)]
\setlength\itemsep{-1pt}
    \item Let $\{z_i\}_{i=1}^m$ denote the intersection points of $L_1'$ with $\Phi_{H_\varepsilon''}(L_1')$ and $\mathcal U$ denotes some small neighbourhood of the southern disc bounded by $L_0$ that contains $\Phi_{H_\varepsilon''}(L_0')$. There is some neighbourhood $V'$ of
    \[V \cup \{z_i\}_{i=1}^m \times \Sym^1(\mathcal U)\]
    such that the isotopy $J^0$ is $(V',\omega_2,j)$-nearly symmetric.
    \item Let 
    \[\widehat{x},\widehat{y} \in CF_*(L_0' \cup L_1', H_\varepsilon'',J^{0,t})\]
    be two capped intersection points $x$ and $y$ such that some curve $u \in \mathcal{M}_1(x,y,H_\varepsilon'',J^{0,t})$ intersects the divisor 
    \[\{\tau_2\} \times \Sym^1(\mathscr{S}_2) \subset \Sym^2(\mathscr{S}_2).\]
    Then the intersection points $x$ and $y$ are of the form $[p'_0,r_1']$ and $[q'_0,r_1']$ respectively, for $p'_0$ and $q'_0$ defined as in Section \ref{subsec.setting and notations} and $r_1' \in L'_1 \pitchfork \Phi_{H_\varepsilon''}(L_1')$. Moreover, for two such $\widehat{x}$ and $\widehat{y}$, such a curve $u$ is unique.
    \item Let 
    \[\widehat{x}\in CF_*(L_0' \cup L_1', H_\varepsilon'',J^{0,t}), \quad \widehat{y} \in CF_*(L_0' \cup L_1', G_\varepsilon'',J^{1,t})\]
    be two capped intersection points $x$ and $y$. Then, there is no curve $u \in \mathcal{M}_0(x,y,H^s,J^{s,t})$ of the continuation map that intersects the divisor
        \[\{\tau_2\} \times \Sym^1(\mathscr{S}_2) \subset \Sym^2(\mathscr{S}_2).\]

    \item All holomorphic curves of index 0 and 1 of the complexes
    \[CF_*(L_0' \cup L_1', H_\varepsilon'',J^{0,t}), \quad CF_*(L_0' \cup L_1', G_\varepsilon'',J^{1,t})\]
    are transversally cut-out. Likewise for all curves of index 0 counted in the continuation map \(\Psi_{(H^s,J^s)}\).
\end{enumerate}
\end{lemma}

\begin{proof}[Proof of Lemma \ref{lem.stretch L0' L1'}]

We view $\mathscr{S}_2$ as the connected sum of $S_0$ and $S_1$, two spheres endowed with almost complex structures $j_0$ and $j_1$ respectively. The sphere $S_i$ contains the Lagrangian $L_i$ and we assume that the connected sum is made away from the support of $H''^s$. Let $\sigma_0$ be the north pole of $S_0$ and $\sigma_1$ be the south pole of $S_1$. The connected sum $\mathscr{S}_2$ of $S_0$ and $S_1$ is obtained by removing two standard holomorphic discs of radii $\rho_0$ and $\rho_1$, $B_{\rho_0}(\sigma_0)$ and $B_{\rho_1}(\sigma_1)$ around $\sigma_0$ and $\sigma_1$ and adding a cylinder in between them. Let $j(T)$ denote the complex structure on $\mathscr{S}_2$ obtained by inserting a standard holomorphic cylinder $[-T,T] \times S^1$ between $S_0$ and $S_1$. We also fix a point $\sigma$ in $\mathscr{S}_2$ contained in the image of this cylinder.

In this way, the symmetric product $(\Sym^2(\mathscr{S}_2),\Sym^2(j(T)))$ admits an open subset holomorphically identified with
\[\Sym^1\big(S_1 - B_{\rho_1}(\sigma_1)\big) \times \Sym^1\big(S_0 - B_{\rho_0}(\sigma_0)\big) \]
endowed with the almost complex structure $j_1 \times j_0$. Moreover, we can actually transfer any path $j_1^{s,t}$ of almost complex structure on $S_1$ as long as $j_1^{s,t}$ is standard on $B_{\rho_1}(\sigma_1)$. More precisely, we can obtain a family $J^{s,t}(T)$ of almost complex structures on $\Sym^2(\mathscr{S_2})$ that satisfies the following constraints. Fix another pair of real numbers $R_1 > \rho_1$ and $R_0 > \rho_0$,
\begin{enumerate}[label=$\bullet$]
\setlength\itemsep{-1pt}
    \item over 
    \[\Sym^1\big(S_1-B_{R_1}(\sigma_1)\big) \times \Sym^1\big(S_0-B_{R_0}(\sigma_0)\big),\]
    the path $J^{s,t}(T)$ agrees with $j_1^{s,t}\times j_0$;
    \item over 
    \[\Sym^1\big(S_1-B_{\rho_1}(\sigma_1)\big)\times \Sym^1\big(B_{R_0}(\sigma_0)- B_{\rho_0}(\sigma_0)\big),\]
    as the normal parameter to $\sigma_0$ goes from $R_0$ to $\rho_0$, $J^{s,t}(T)$ splits as a product and connects $j_1^{s,t}\times j_0$ to $j_1\times j_0$;
    \item over the rest, $J^{s,t}(T)=\Sym^2\big(j(T)\big)$.
\end{enumerate}

Since $j_1^{s,t}\equiv j_1$ near $\sigma_1$ the above properties fit together and we point out that for now the property $(i)$ is verified. We want to study the holomorphic curves in the Floer complexes as the neck parameter $T$ goes to $+\infty$. Fix a class $\phi \in \pi_2(x,y)$ and assume that there is a strip $u \in \mathcal{M}_1(x,y, H_\varepsilon'',J^{0,t}(T))$ or $u \in \mathcal{M}_0(x, y, H^s, J^{s,t}(T))$ that represents this class and intersects the divisor $\{\tau_2\} \times \Sym^1(\mathscr{S}_2)$. We distinguish two cases. Either $n_{\sigma}(\phi)=0$ or $n_{\sigma}(\phi)=n > 0$. 

\medskip

\noindent \textbf{Case 1:} If $n_{\sigma}(\phi)=0$ and since the almost complex structures $J^{s,t}(T)$ split outside of the gluing area, we get that for $T$ large enough the curve $u$ splits as the product of two curves $u_1 \times u_0$, where $u_0$ is a $j_0$-holomorphic curve in $S_0-B_{R_0}(\sigma_0)$ and $u_1$ is a $j_1^{s,t}$-holomorphic curve in $S_1-B_{R_1}(\sigma_1)$.

Moreover, if $u \in \mathcal{M}_1(x,y, H_\varepsilon'',J^{0,t}(T))$, then $\mu(u)=\mu(u_0)+\mu(u_1)=1$. From transversality, we have $\mu(u_0), \mu(u_1) \geq 0$ with equality only if the curve is constant. This implies that one of those curve has Maslov index 0, hence it is constant. The curve $u_0$ is not constant, since otherwise $u$ would not intersect 
\[\{\tau_2\} \times \Sym^1(\mathscr{S}_2) \subset \Sym^2(\mathscr{S}_2),\]
a contradiction. Thus, the curve $u_1$ is constant equal to some $r_1'\in L'_1 \pitchfork \Phi_{H_\varepsilon''}(L_1')$ and the curve $u_0$ is the unique holomorphic strip of index 1 between the two points $p'_0$ and $q'_0$ defined as in Section \ref{subsec.setting and notations} passing through the south pole $\tau_2$. This implies that the curve $u$ is of the form $\{r_1'\} \times u_0$, moreover, such a curve is unique if we want to have it with fixed extremities and such that it intersects
\[\{\tau_2\} \times \Sym^1(\mathscr{S}_2) \subset \Sym^2(\mathscr{S}_2).\]
Since this curve can be written as the product of two transverse curves, the curve itself is transverse.

If $u \in \mathcal{M}_0(x, y, H^s, J^{s,t}(T))$, then $\mu(u)=\mu(u_0)+\mu(u_1)=0$, and therefore $\mu(u_0)=\mu(u_1)=0$. But since the isotopy $H^s$ is constant on $S_0-B_{R_0}(\sigma_0)$, the strip $u_0$ must be constant at an intersection point and cannot pass through $\tau_2$. Hence, the case $n_{\sigma}(\phi)=0$ can be ruled out for $u \in \mathcal{M}(x, y, H^s, J^{s,t}(T))$ intersecting $\{\tau_2\} \times \Sym^1(\mathscr{S}_2)$.

\medskip

\noindent \textbf{Case 2:} If $n_{\sigma}(\phi)=n > 0$, we will show by contradiction that if $T$ is sufficiently large, there is no strip $u \in \mathcal{M}(x,y, H_\varepsilon'',J^{0,t}(T))$ of index 0 or 1, or $u \in \mathcal{M}_0(x, y, H^s, J^{s,t}(T))$ that represents $\phi$ and that intersects the divisor $\{\tau_2\} \times \Sym^1(\mathscr{S}_2)$. Let us assume by contradiction that for some sequence $T \to +\infty$ we can find a $J^{s,t}(T)$-holomorphic curve
\[u_T \in \mathcal{M}(x, y, H''_\varepsilon, J^{0,t}(T))\vspace{-0.5mm}\]
of index 0 or 1 or a $J^{s,t}(T)$-holomorphic curve
\[u_T \in \mathcal{M}_0(x, y, H^s, J^{s,t}(T))\vspace{-0.5mm}\]
of the continuation map that represents the class $\phi$. As in \cite{OS} there is a uniform energy estimate on the energy of such curves, and by the Gromov compactness theorem we obtain a bubble tree
\[u_\infty \in \mathcal{M}(\Sym^2(S_0 \vee S_1))\vspace{-0.5mm}\]
consisting of a possibly broken strip, disc bubbles, and sphere bubbles. That is, we can write $u_\infty$ as the product of two holomorphic curves $u_1$ and $u_0$ in $(S_0, j_0)$ an $(S_1, j^{s,t}_1)$ respectively as well as some disc bubbles and sphere bubbles. Assume for now that there is no disc or sphere bubble. Then we have $n_{\sigma_0}(u_0)=n_{\sigma_1}(u_1)$ and from a virtual index computation, see Remark \ref{rk.first vdim computation}, we obtain $\mu(u_\infty) = \mu(u_0)+ \mu(u_1)-2 n_{\sigma_0}(u_0)$. Now in $S_0$ we have a complete understanding of the holomorphic curves of all indices. In particular, for curves passing through $\tau_2$, we have \(\mu(u_0) \geq 1+ 2 n_{\sigma_0}(u_0)\) and in turn
\[\mu(u_\infty)=\mu(u_1)+\mu(u_0)-2n_{\sigma_0}(u_0) \geq \mu(u_1)+1.\]
Since by hypothesis $\mu(u_\infty)$ is 0 or 1 and since $\mu(u_1) \geq 0$ this directly rules out the case $\mu(u_\infty)=0$ and forces $\mu(u_1)=0$ hence the strip $u_1$ is a constant strip and thus $n_{\sigma_1}(u_1)=0$, this is a contradiction with the above assumption. In case of bubbling, we use the fact that adding a sphere or disc bubble to $u_\infty$ would increase its Maslov index by at least $2$, contradicting $\mu(u_\infty)=0$ or $1$. Note that since these curves do not exist, we have automatically transversality for them.

\medskip

We claim that there is only a finite number of homotopy classes $\phi$ we need to consider. Indeed, the classes have Maslov index 0 or 1, and their energy is determined by the end points, therefore the claim follows from Gromov compactness.
Hence, from the case disjunction, we conclude that there exists a real number $T$ and an almost complex structure $J^{s,t}(T)$ such that all the curves in the Floer complex and in the continuation maps are as in $(ii)$ and $(iii)$. Now in order to obtain the transversality as claimed in $(iv)$ and for almost complex structures as in $(i)$, we just need to do a small perturbation of $J^{s,t}$ in the suitable class of almost complex structures. Indeed, in this case, all the curves as in $(ii)$ and $(iii)$ are already transverse hence small perturbations will not change those moduli spaces. 

\medskip

We first show that $J^0=J^{0,t}$ can be chosen as $(V',\omega,j(T))$-nearly symmetric. We proceed as in Section \ref{subsec.some trans results}. For any intersection point $z_i$, any $J^0$-holomorphic curve $u$ of the Floer complex that intersects a small neighbourhood of $\{z_i\} \times \Sym^1(\mathcal U)$ must either leave it or be the product of a $j_0$-holomorphic curve in $\Sym^1(\mathcal U)$ with the constant curve $\{z_i\}$. If it leaves the neighbourhood, then there is no transversality issue, and if it is as in the second case, it is automatically transverse as the product of two transverse curves. The case of the strips of the continuation maps is exactly as in Section \ref{subsec.some trans results}.
\end{proof}

\begin{rk}\label{rk.first vdim computation}We give the details of the index computation in the above proof. Let $\ell \coloneqq n_{\sigma_0}(u_0)=n_{\sigma_1}(u_1)$, and $\mathcal M_\ell$ be the moduli space of tuples $(v_1,v_0,y_1,\ldots,y_\ell)$ where 
\[(v_1,v_0):\R\times [0,1]\to S_1 \times S_0\] 
is $j_1^{s,t}\times j_0$-holomorphic and satisfy the Lagrangian boundary condition, $[v_i]=[u_i]$, and $(y_1,\ldots,y_\ell)$ is a $\ell$-tuple of points of $\R\times (0,1)$. Then, $\mathcal M_\ell$ has virtual dimension 
\[\text{vdim} (\mathcal M_\ell) = \mu(u_1)+\mu(u_0)+2\ell.\]
For $1\leq i\leq \ell$, let 
\begin{align*}
ev_i& :\mathcal M_\ell~\to~ S_1 \times S_0\\
&(v_1,v_0,y_1,\ldots,y_\ell)\mapsto (v_1(y_i), v_0(y_i))
\end{align*}
be the evaluation map at $y_i$. Then, $u_\infty$ determines a unique element 
\[(u_1,u_0,z_1,\ldots,z_\ell)\in \mathcal M\coloneqq \bigcap\limits_{1\leq i \leq \ell}ev_i^{-1}(\{(\sigma_1,\sigma_0)\}),\] 
up to permutation of the $z_i$, where $\{z_1,\ldots,z_\ell\}=u_0^{-1}(\{\sigma_0\})=u_1^{-1}(\{\sigma_1\})$. Since $\{(\sigma_1,\sigma_0)\}\subset S_1\times S_0$ has codimension $4$, we get as stated above 
\[\mu(u_\infty)=vdim (\mathcal M)=vdim (\mathcal M_\ell)-4\ell=\mu(u_1)+\mu(u_0)-2\ell.\]
\end{rk}

We now write and prove the following lemma. Recall that the isomorphism of filtered module $g$ is defined in Proposition \ref{prop.iso of filtered module S1 S2 and S2}.

\begin{lemma}\label{lem.stretch L0 L2}
Let $J^{s,t}$ be the isotopy of family of almost complex structure on $\Sym^2(\mathscr{S}_2)$ obtained from Lemma \ref{lem.stretch L0' L1'} and let $j_1$ be an almost complex structure on $\mathscr{S}_1$. Then there exists a complex structure $j'$ on $(\mathscr{S}_3,\omega_3)$ and a neighbourhood $V$ of the diagonal $\Delta$ such that there is some isotopy $(\widehat{J}^s)_{s \in [0,1]} \coloneqq (\widehat{J}^{s,t})_{s,t \in [0,1]}$ of families of $(V,\omega_3, j')$-nearly symmetric almost complex structure that verifies the following properties.
\begin{enumerate}[label=(\roman*)]
    \item Let
    \[\widehat{y} \in CF_*(L_0' \cup L_1', H_\varepsilon'',J^{s,t})\]
    be a capped intersection that corresponds to the intersection point $y$, and $p_2'$ the representative of the fundamental class of $HF(L_2',H'_\varepsilon)$ as defined in Section \ref{subsec.setting and notations}. Then,
    \[\partial g( \widehat{y}\otimes p_2')=g(\partial\widehat{y}\otimes p_2'+\widehat{y} \otimes \partial p_2'),\]
    where $\partial$ denotes the differential in the complexes $CF_*(L_0 \cup L_1 \cup L_2, H_\varepsilon,\widehat{J}^{0,t})$, $CF_*(L_0' \cup L_1', H_\varepsilon'',J^{0,t})$ and $CF_*(L_2', H_\varepsilon',j_1)$.
    \item Let $\widehat{y}$ be as in $(i)$, then the continuation maps commute with $g$ for $\widehat{y}\otimes p_2'$,
    \[\Psi_{(H^s,\widehat{J}^s)}\big(g(\widehat{y} \otimes p_2')\big)=g\big(\Psi_{(H''^s,J^s)}(\widehat{y}) \otimes \Psi_{(H'^s,j_1)}(p_2')\big).\]
    \item All holomorphic strips of index 0 and 1 of the complexes
    \[CF_*(L_0 \cup L_1\cup L_2, H_\varepsilon,\widehat J^{0,t}), \quad CF_*(L_0 \cup L_1 \cup L_2, G_\varepsilon,\widehat J^{1,t})\]
    are transversally cut-out. Likewise for all holomorphic strips counted in the continuation map \(\Psi_{(H^s,\widehat{J}^s)}\).
\end{enumerate}
\end{lemma}

\begin{proof}[Proof of Lemma \ref{lem.stretch L0 L2}]
The proof is rather long, we have split it into different parts. The parts are separated by titles in bold font.

\bigskip

\noindent \textbf{Set-up of the neck-stretching.} We proceed in a similar manner as in Lemma \ref{lem.stretch L0' L1'}. We construct the almost complex structure $\widehat{J}^{s,t}$ as a sufficiently stretched version of an almost complex structure on $\mathscr{S}_3$ viewed as the connected sum of $\mathscr{S}_2$ and $\mathscr{S}_1$. We will then analyse the strips starting from $g(y \otimes p'_2)$. Let $j$ be the complex structure on $\mathscr{S}_2$ given by Lemma \ref{lem.stretch L0' L1'} and $j_1$ any complex structure on $\mathscr{S}_1$. We denote $\tau_2$ the south pole of $\mathscr{S}_2$ and $\tau_1$ the north pole of $\mathscr{S}_1$. We remove two standard holomorphic discs $B_{\rho_2}(\tau_2)$ and $B_{\rho_1}(\tau_1)$ around $\tau_2$ and $\tau_1$ outside the support of $H''_\varepsilon$ and $H_\varepsilon'$ and add a standard cylinder $[-T,T] \times S^1$. We also fix a point $\tau$ in $\mathscr{S}_3$ contained in the image of this cylinder. We identify the support of $\psi$ in $\mathscr{S}_3$ and its image in $\mathscr{S}_2$ as well as the support of $\theta$ in $\mathscr{S}_3$ with its image in $\mathscr{S}_1$. This construction gives us a complex structure $j'(T)$ on $\mathscr{S}_3$.

\smallskip

An open subset of $(\Sym^3(\mathscr{S}_3),\Sym^3(j'(T)))$ is holomorphically identified with 
\[\Sym^2\big(\mathscr{S}_2-B_{\rho_2}(\tau_2)\big) \times \Sym^1\big(\mathscr{S}_1-B_{\rho_1}(\tau_1)\big)\vspace{-0.5mm}\]
endowed with the almost complex structure $\Sym^2(j)\times \Sym^1(j_1)$. As in Lemma \ref{lem.stretch L0' L1'} we can transfer the path of almost complex structure $J^{s,t}$ on $\Sym^2(\mathscr{S})$ to construct $\widehat{J}^{s,t}$ in the following way. Fix another pair of real numbers $R_2> \rho_2$ and $R_1 > \rho_1$,
\begin{itemize}
\setlength\itemsep{-1pt}
    \item over 
    \[\Sym^2\big(\mathscr{S}_2-B_{R_2}(\tau_2)\big)\times \Sym^1\big(\mathscr{S}_1-B_{R_1}(\tau_1)\big),\]
    $\widehat{J}^{s,t}(T)$ agrees with $(\Sym^2\psi)^*J^{s,t}\times j_1$;
    \item over 
    \[\Sym^2\big(\mathscr{S}_2-B_{\rho_2}(\tau_2)\big)\times \Sym^1\big(B_{R_1}(\tau_1)- B_{\rho_1}(\tau_1)\big),\]
    as the normal parameter to $\tau_1$ goes from $R_1$ to $\rho_1$, $\widehat{J}^{s,t}(T)$ splits as a product and connects $(\Sym^2\psi)^*J^{s,t}\times j_1$ to $\Sym^2(j)\times j_1$;
    \item $\widehat{J}^{s,t}(T)=\Sym^3\big(j'(T)\big)$ everywhere else.
\end{itemize}
Since $J^{s,t}$ is equal to $\Sym^2(j)$ on a neighbourhood of $\Sym^1(\mathscr{S}_2) \times \{\tau_2\}$ all those conditions fit together continuously. Moreover, if we denote $z_1,\ldots ,z_m$ the intersection points between $L_1$ and $\Phi_{H_\varepsilon}^1(L_1)$ and denote $\mathcal{V} \subset \mathscr{S}_3$ the gluing of the open set $\mathcal{U} \subset \mathscr{S}_2$ defined in $(i)$ of Lemma \ref{lem.stretch L0' L1'} with $\mathscr{S}_1$ along the cylinder $[-T,T]\times S^1$. We know from Lemma \ref{lem.stretch L0' L1'} that the almost complex structure $J^{0,t}(T)$ is symmetric on a neighbourhood of $\{z_i\} \times \Sym^1(\mathcal U)$ for all $i=1, \ldots, m$. This means that in the first bullet point on a neighbourhood of $\{z_i\} \times \Sym^2(\mathcal V)$, the almost complex structure satisfies
\[\widehat{J}^{0,t}(T)= (\Sym^2 \psi)^* J^{0,t} \times j_1= \Sym^3(j'(T)).\vspace{-0.5mm}\]
This means that in the second bullet point the interpolation can be chosen to be trivial on a neighbourhood of $\{z_i\} \times \Sym^2(\mathcal V)$. This isotopy $\widehat{J}^{0,t}(T)$ can be therefore chosen to be symmetric on a neighbourhood of $\{z_i\} \times \Sym^2(\mathcal V)$ for all $i =1, \ldots, m$.

\bigskip

\noindent \textbf{Description of what we want to prove.} To prove $(i)$ and $(ii)$, we use again a neck-stretching argument by letting the parameter $T$ go to infinity. After stretching sufficiently the almost complex structure we will manage to count the curves appearing in $(i)$ and $(ii)$ of the statement of the lemma. We describe now the end result of this curve count and how it implies $(i)$ and $(ii)$. The rest of the proof is dedicated to the proof of the announced count. We show in the end of the proof that all curves counted in \vspace{-0.5mm}
\[\Psi_{(H^s,\widehat{J}^s)}(g(\widehat{y}\otimes p_2'))\vspace{-0.5mm}\]
arise as the product of a strip counted in \vspace{-0.5mm}
\[\Psi_{(H''^s,J^s)}(\widehat{y})\vspace{-0.5mm}\]
and the constant strip $p_2$, with a one-to-one correspondence. This yields $(ii)$. Meanwhile, the curves counted in $\partial g( \widehat{y}\otimes p_2')$ arise as the product of a curve counted in $\partial p_2'$ with a constant curve at $\Sym^2(\psi)^{-1}(y)$ or as a product of a constant curve at $p_2$ with a curve counted in $\partial \widehat{y}$ the only exception being all curves that are product of a constant curve at $p_2$ with a curve counted in $\partial \widehat{y}$ of the same form as the curves in $(iii)$ of Lemma \ref{lem.stretch L0' L1'}. These curves correspond (in the sense that the endpoints correspond exactly) to curve with domain a cylinder. We show then that the signed count of those cylinders is the same as the count of such curves. This implies $(i)$.

\bigskip

\noindent \textbf{Start of the case analysis of the curves in the stretched almost complex structure.} We denote $x$ the intersection point associated with the generator $g(y \otimes p_2')$ and let $z$ be a generator of 
\[CF_*(L_0 \cup L_1 \cup L_2, H_\varepsilon'') \text{ or } CF_*(L_0 \cup L_1 \cup L_2, G_\varepsilon).\]
We fix a class $\phi \in \pi_2(x,z)$ and study holomorphic strips $u \in \mathcal{M}(x,z,H_\varepsilon,\widehat{J}^{0,t}(T))$ of index 0 or 1 or $u \in \mathcal{M}(x,z,H^s,\widehat{J}^{s,t}(T))$ of index 0 representing this class as the neck parameter $T$ goes to infinity. Recall that we fixed a point $\tau$ in the inserted neck $[-T,T]\times S^1$. We distinguish two cases: $n_\tau(\phi)=0$ or $n_{\tau}(\phi)=n >0$.

\medskip

\noindent \textbf{Case 1: If $n_{\tau}(\phi)=0$.} 

\noindent Since the almost complex structures split outside of the gluing area, we get that, as $T$ goes to infinity, the strip $u$ can be written as the product of two strips $u = u_2 \times u_1$ with $u_2$ a $J^{s,t}$-holomorphic strip in $\Sym^2\big(\mathscr{S}_2- B_{\rho_2}(\tau_2)\big)$ and $u_1$ a $j_1$-holomorphic strip in $\Sym^1\big(\mathscr{S}_2-B_{\rho_1}(\tau_1)\big)$. Let us first assume that $u$ is a curve in $\mathcal{M}(x,z,H_\varepsilon,\widehat{J}^{0,t}(T))$ of index 1. Since $1=\mu(u)=\mu(u_2)+\mu(u_1)$, at least one of the two maps $u_2$ or $u_1$ is constant and the other is a holomorphic strip such that $n_{\tau_2}(u_2)=0$ and $n_{\tau_1}(u_1)=0$. This implies that there is a one-to-one correspondence between the $\widehat{J}^{s,t}$-holomorphic strips $u$ starting from $x$ such that $n_{\tau}(u)=0$ and the $J^{s,t}$-holomorphic (resp. $j_2$-holomorphic) strips $u_2$ (resp. $u_1$) starting form $y$ (resp. $p_2'$) such that $n_{\tau_2}(u_2)=0$ (resp. $n_{\tau_1}(u_1)=0$). When $u$ is a curve that contributes to the continuation map, it will also split as a curve $u_2 \times u_1$. However, this time, since the isotopy $H^s$ is constant on $L_1$, the strip $u_1$ is constant. In Case 2, we will show that there are no other strips contributing to the continuation map starting from $x$ with $n_\tau(\phi)>0$, so this implies readily that 
\[\Psi_{(H^s,\widehat{J}^s)}\big(g(\widehat{y} \otimes p_2')\big)=g\big(\Psi_{(H''^s,J^s)}(\widehat{y}) \otimes \Psi_{(H'^s,j_1)}(p_2')\big).\]

\medskip

\noindent \textbf{Case 2: If $n_\tau(\phi)=n >0$.} 

\noindent We first assume that $u$ is a strip counted in the continuation map. We want to show that as $T$ goes to $+\infty$ no such strips persist. In fact, let us assume by contradiction that for $T$ going to $+\infty$ we get a strip $u_T \in \mathcal{M}(x,z,H^s,\widehat{J}^{s,t}(T))$ that represents the class $\phi$. As in the proof of Lemma \ref{lem.stretch L0' L1'} we consider $u_\infty$ the Gromov limit of $u_T$. We obtain a bubble tree \(u_\infty \in \mathcal{M}\big(\Sym^3(\mathscr{S}_2 \vee \mathscr{S}_1)\big)\) consisting of possibly broken strips, disc bubbles and sphere bubbles. From the same reasons as in Lemma \ref{lem.stretch L0' L1'} and Remark \ref{rk.first vdim computation} we can rule out any kind of bubbling. Observe that
\[\Sym^3(\mathscr{S}_2 \vee \mathscr{S}_1)=\bigcup_{i=0}^3\Sym^i(\mathscr{S}_2)\times \Sym^{3-i}(\mathscr{S}_1)\]
where the intersections between the different terms correspond to the singular locus. Since $u_\infty$ has boundary in $\Sym^2(\mathscr{S}_2)\times \Sym^1(\mathscr{S}_1)$ and has no nodal point, it has to be contained completely in it. Thus $u_\infty=u_2 \times u_1$ with $u_2$ contained in $\Sym^2(\mathscr{S}_2)$ and $u_1$ in $\Sym^1(\mathscr{S}_1)=\mathscr{S}_1$. By a virtual index computation, we get 
\[0=\mu(u_\infty)\geq \mu(u_2)+ \mu(u_1)-2 n_{\tau_1}(u_1).\]
Note that since $u_1$ is a holomorphic strip in $\mathscr{S}_1$ starting from $p_2'$, we have $\mu(u_1)\geq 2 n_{\tau_1}(u_1)$. Together with the above inequality, this implies that $u_2$ has index 0. However, from Lemma \ref{lem.stretch L0' L1'}, there is no such curve $u_2$ with $n_{\tau_2}(u_2)>0$. This is the contradiction we wished for. This concludes the part of the lemma that concerns the curves in the continuation maps.

\medskip

\noindent We now assume that $u$ is a strip in $\mathcal{M}(x,z,H_\varepsilon,\widehat{J}^{0,t}(T))$. We want to show that any such strip has domain a cylinder in between $L_0$ and $L_2$ as in Section \ref{subsubsec.count sphere} in particular $z$ is always an intersection that corresponds to $g(w \otimes p_2')$ for some intersection point $w$ in $\mathscr{S}_2$. Let us then assume by contradiction that as $T$ goes to $+\infty$, there exists always a curve $u_T \in \mathcal{M}(x,z,H_\varepsilon,\widehat{J}^{0,t}(T))$ whose domain is not a cylinder. Since there are finitely many homotopy classes of strips of index 0 or 1 with fixed end points, up to passing to a subsequence, we can assume that the curves $u_T$ are all in the same homotopy class, which by assumption is different from the two classes corresponding to cylinders described before. We consider a bubble tree $u_\infty \in \Sym^3(\mathscr{S}_2 \vee \mathscr{S}_1)$ a Gromov limit of the sequence $u_T$. Then $u_\infty$ is the product strip $u_2 \times u_1 \in \Sym^2(\mathscr{S}_2) \times \Sym^1(\mathscr{S}_1)$ with some eventual disc and sphere bubbles. From a virtual index computation as in Remark \ref{rk.first vdim computation} we can exclude disc and sphere bubbles and moreover, $u_\infty$ is just the product of the two strips $u_1$ and $u_2$ with
\[\mu(u_\infty)=\mu(u_1) + \mu(u_2)- 2n_{\tau_1}(u_1).\vspace{-0.5mm}\]
The strip $u_1$ is a $j_1$-holomorphic strip in $\mathscr{S}_1$ starting at $p_2'$, therefore we must have $\mu(u_1) \geq 2n_{\tau_1}(u_1)$, which implies that $\mu(u_2)\leq \mu(u_\infty)\leq 1$.

\smallskip

Since $n_{\tau_2}(u_2)>0$, we get from Lemma \ref{lem.stretch L0' L1'} that the $J^{s,t}$-holomorphic strip $u_2$ is a strip from $[p_0',r_1']\coloneqq y$ to $[q_0',r_1']\coloneqq w$ for some $r_1'\in L'_1 \pitchfork \Phi_{H_\varepsilon''}(L_1')$. Moreover, this curve is unique and is the product of a constant curve at $r_1'$ and a holomorphic strip from $p_0'$ to $q_0'$. In particular, $n_{\tau_1}(u_1)=n_{\tau_2}(u_2)=1$, which implies that $u_1$ is a Maslov index two strip from $p_2'$ to itself, i.e. a disc with boundary on $L'_2$ or $\Phi_{H_\varepsilon'}(L'_2)$ with a slit along the other Lagrangian. The domain of $u_\infty$ is thus the domain of a stretched cylinder. This implies that for $T$ large enough, all strips $u \in \mathcal{M}(x,z,H_\varepsilon,\widehat{J}^{0,t}(T))$ of index smaller than 1 and with $n_\tau(u)\neq 0$ have domain a cylinder.

Let $T$ be such real number. Then from the definition of the almost complex structure, $\widehat{J}^{0,t}(T)$ is symmetric on a neighbourhood of $\{z_i\}\times \Sym^2(\mathcal V)$, where $z_i$ is any intersection point of $L_1$ and $\Phi_{H_\varepsilon}(L_1)$. Thus, any curves whose domain is a cylinder are completely contained in $\{z_i\} \times \Sym^2(\mathcal V)$, where $z_i=r_1\coloneqq \psi^{-1}(r_1')$ is the intersection point that corresponds to $r_1'$. This implies that those curves and their orientation behave exactly as in Lemma \ref{lem.count of cylinders as in OS}.

\bigskip

\noindent \textbf{Analysis of the sign of the cylinders obtained in Case 2.} To finish the proof of $(ii)$, we want to show that the (signed) count of the cylinders from $[q_0,r_1,p_2']$ to $[p_0,r_1,p_2']$ is equal to the signed count of the unique curve that is as in $(ii)$ of Lemma \ref{lem.stretch L0' L1'} going from $[q_0',r_1']$ to $[p_0',r_1']$. Note that once 
\[j'(T), \, L_0, \, \Phi^1_{H_\varepsilon}(L_0), \, L_1 \text{ and }\Phi^1_{H_\varepsilon}(L_2)\] 
are fixed, the orientation data of curves of cylindrical domain is fixed as long as $\widehat{J}^{0,t}(T)=\Sym^3(j'(T))$ along those curves. In particular, this allows us to change $\Phi^1_{H_\varepsilon}(L_1)$ replacing it by $\Phi^1_{G_\varepsilon}(L_1)$. In this case, we have simpler complexes
\[CF_*(L_0 \cup L_1 \cup L_2, G_\varepsilon, \widehat{J}^{0,t}), \, CF_*(L_0 \cup L_1, G_\varepsilon'', J^{0,t}) \text{ and } CF_*(L_2, G_\varepsilon', j_1).\] 
Indeed, they have rank $2^3$, $2^2$ and $2^1$ respectively as $\Lambda$-module. Since the associated homology has the same rank over $\Lambda$ as the underlying complex and since $\Lambda$ is an integral domain, the differential of the complexes has to vanish, see Lemma \ref{lem.vanishing of partial}. In particular, $g$ is an isomorphism and
\begin{equation}\label{eq.hard formula in simpler case}
	\partial g(\widehat{y} \otimes p_2')=g(\partial \widehat{y} \otimes p_2' + \widehat{y} \otimes \partial p_2')
\end{equation}
holds automatically. However, as we have seen just before, this identity comes from a one-to-one correspondence between the curves counted on each side of this equality, apart from the count of cylinders and the corresponding curve as in $(iii)$ of Lemma \ref{lem.stretch L0' L1'}. From the correctness of the formula (\ref{eq.hard formula in simpler case}) in this case, the signed count of curves has to be the same. Hence, it is also the same in our more general case.

\bigskip

\noindent \textbf{Proof that one can obtain conclusions $(i)$ and $(ii)$ with transversality.} Now we prove the last point of the lemma. Since our only conclusions concern only some specific curves, we only need to take care of the transversality of any curve as in $(i)$ and $(ii)$. The perturbations needed for any other curve does not change our result. Any curve that is as in Case 1 can be easily dealt with, indeed they are product of holomorphic strips that are already transverse from the definition of $J^{s,t}$ and $j_1$. In the second case, we have curves whose domain are cylinders and also that admit a $2$-fold branched covering by a $j(T)$-holomorphic cylinder in $\mathscr{S}_3$. After some small perturbation of the Lagrangians $L_0$ and $L_1$, the count is exactly $\pm 1$ as previously stated with the correct sign to make the statement of the lemma correct. This finishes the proof of the transversality and thus of the lemma.
\end{proof}

Proposition \ref{prop.maps of fund class is fund class} follows from Lemma \ref{lem.stretch L0 L2}.

\begin{proof}[Proof of Proposition \ref{prop.maps of fund class is fund class}]
From Lemma \ref{lem.stretch L0 L2} $(i)$ and by the linearity of the differential and of the map $g$, for all the chain elements $y \in CF_*(L_0' \cup L_1',H'_\varepsilon)$,
\[\partial g( y\otimes p_2')=g(\partial y\otimes p_2'+y \otimes \partial p_2')=g(\partial y\otimes p_2'),\]
where the last inequality comes from the fact that $\partial p_2'=0$. 
Let $x\in CF_*(L_0' \cup L_1',H'_\varepsilon)$ be a representative of the fundamental class. Since $x$ is a cycle, then $g(x \otimes p_2')$ is also a cycle from the above formula. Moreover, from Lemma \ref{lem.stretch L0 L2} $(ii)$, 
\[\Psi_{(H^s,\widehat{J}^s)}\big(g(x\otimes p_2')\big)=g\big(\Psi_{(H''^s,J^s)}(x) \otimes \Psi_{H'^s,j_2}(p_2')\big).\]
Hence, at the homology level, the continuation map $\Psi_{(H^s,\widehat{J}^s)}$ sends the class $[g(x \otimes p_2')]$ to the fundamental class of
\[HF_*(L_0\cup L_1 \cup L_2, H_\varepsilon).\]
This finishes the proof of Proposition \ref{prop.maps of fund class is fund class}.
\end{proof}

\section{Some consequences and discussions}\label{sec:consequences}

As stated in the Introduction, the main consequence of Theorem \ref{thm.main thm} is Corollary \ref{coro:alternative}, which states that if, for some $\lambda\in[\frac 1 3,\frac 1 2)$, the quasimorphism $r_\lambda$ does not vanish identically, then the Hofer distance on the space of equators of the sphere is unbounded. We give some necessary conditions on a Hamiltonian diffeomorphism $\varphi$ for which one could have $r_\lambda(\varphi)\neq 0$, and deduce some additional consequences of the existence of such a diffeomorphism.

\begin{coro}
For all $\lambda\in[\frac 1 3,\frac 1 2)$, the quasimorphism $r_\lambda$ vanishes on Hamiltonian diffeomorphisms with no topological entropy.
\end{coro}

\begin{proof}
    Let $\varphi$ be a Hamiltonian diffeomorphism of $S^2$. If for some $\lambda$, $r_\lambda(\varphi)$ is non-zero, then by Proposition \ref{prop:Khanevsky}, $d(\varphi^n(L_0),L_0)$ grows linearly with $n$. However, by \cite[Proposition 3.3.10]{HumHDR}, if $h_{\text{top}}(\varphi)=0$, then $d(\varphi^n(L_0),L_0)$ grows sub-linearly, and therefore one must have $r_\lambda(\varphi)=0$.
\end{proof}

\begin{rk}\label{rk:vanish_on_autonomous}
    In particular, since autonomous Hamiltonian diffeomorphisms of surfaces have no topological entropy, we recover that $r_\lambda$ vanishes on autonomous Hamiltonian diffeomorphisms of $S^2$, which was already shown in \cite{BFPS} using other methods.
\end{rk}

We can then refine Corollary \ref{coro:alternative}.
\begin{coro}
    At least one of the following is true:
    \begin{itemize}
        \item for all $\lambda\in[\frac 1 3,\frac 1 2)$, $r_\lambda$ vanishes identically on $\Ham(S^2)$;
        \item the Hofer distance on the space of equators of the sphere is unbounded, and there are Hamiltonian diffeomorphisms arbitrarily far away, in the Hofer distance, from the set of autonomous Hamiltonian diffeomorphisms of the sphere.
    \end{itemize}
\end{coro}

\begin{proof}
    It is enough to show that if $r_\lambda(\varphi)\neq 0$, then the Hofer distance from $\varphi^n$ to the set $\Aut(S^2)$ of autonomous Hamiltonian diffeomorphisms of the sphere grows linearly in $n$. Using that $r_\lambda$ is a Hofer-Lipschitz and homogeneous quasimorphism that vanishes on $\Aut(S^2)$, we get for $\psi\in\Aut(S^2)$, that
    \begin{align*}
        d(\varphi^n,\psi)&=\Vert \varphi^n\psi^{-1}\Vert \geq K^{-1}\vert r_\lambda(\varphi^n\psi^{-1})\vert\\
        &\geq K^{-1}\left(\vert r_\lambda(\varphi^n)\vert+\vert r_\lambda(\psi^{-1})\vert-D\right)\\
        &\geq K^{-1}\left(n\vert r_\lambda(\varphi)\vert-D\right),
    \end{align*}
    where $K$ is the Lipschitz constant of $r_\lambda$ and $D$ its defect. Taking the infimum over $\psi$, and since $r_\lambda(\varphi)\neq 0$, we get that $d(\varphi^n,\Aut(S^2))$ grows linearly in $n$, which concludes the proof.
\end{proof}

We now discuss what a possible candidate $\varphi$ such that $r_\lambda(\varphi)\neq 0$ would look like. From what was stated before, $\varphi$ would need to have entropy. Moreover, by \cite[Theorem 8]{K09}, the number of intersection points between $\varphi^n(L_0)$ and $L_0$ would have to grow exponentially. Therefore, computing the Floer homology of iterations of $\varphi$ (which one needs to do in principle to compute homogenised spectral invariants) seems to be very difficult. However, recall that the difference between a quasimorphism and its homogenisation is bounded by the defect of the quasimorphism. In our case, using the bounds on the spectral norm from \cite{KS}, we have that the defect of $c_{\underline L'}$ is bounded by $3\lambda$, that of $c_{\underline L}$ by $2\lambda$, and that of  $c_{L_0}$ by $\frac 1 2$. Since those three quasimorphisms are subadditive, we have that the defect of $c_{\underline L'}-c_{\underline L}-c_{L_0}$ can be bounded by $\max\{3\lambda,2\lambda+\frac 1 2\}=2\lambda+\frac 1 2$. Then, it is enough to find a Hamiltonian $H$ such that $\vert(c_{\underline L'}-c_{\underline L}-c_{L_0})(H)\vert>2\lambda+\frac 1 2$ to know that $r_\lambda(\Phi^1_H)\neq 0$. We do expect that the non-homogenised difference of spectral invariants $(c_{\underline L'}-c_{\underline L}-c_{L_0})$ does not vanish identically, however, for all the examples of Hamiltonians for which we conjecture a non-zero value, it would still stay bounded by $\lambda$.

On the other hand, one cannot rely on a filtered (chain) isomorphism as we did in this paper (cf Proposition \ref{prop.iso of filtered module S1 S2 and S2}) to show that $r_\lambda$ vanishes identically, since for Hamiltonian diffeomorphisms creating intersections between some component of the link and the image of a different component, the two complexes would not even have the same number of generators. 

We ask the following related question:

\begin{ques}
    Does there exist a Lagrangian $L$ inside $S^2$, Hamiltonian isotopic to the equator, but that is not the image of the equator by any autonomous Hamiltonian diffeomorphism?
\end{ques}

This question is equivalent to ask if there are Hamiltonian diffeomorphisms of $S^2$ that cannot be decomposed as a composition of an autonomous Hamiltonian diffeomorphism with a Hamiltonian diffeomorphism fixing the equator (as a set). A negative answer would imply that our quasimorphism $r_\lambda$ vanishes identically.

Similar questions can be asked about some fragmentation norms on $\Ham(S^2)$. In what follows, we assume that the area form on $S^2$ is normalised so that its total area is $1$.

\begin{dfn}
For $A\in(0,1)$ and $\varphi\in \Ham(S^2)$, the $A$-fragmentation norm of $\varphi$ is defined as
\[\Vert \varphi\Vert_A=\min\{n,\exists \varphi_1,\ldots,\varphi_n \text{ supported in discs of area}\leq A, \varphi=\varphi_1\circ\ldots\circ\varphi_n\}.\]
\end{dfn}

This norm is well-defined since $\Ham(S^2)$ is simple by Banyaga's theorem \cite{Banyaga}, and the subgroup generated by Hamiltonian diffeomorphisms supported in discs of area smaller or equal to $A$ is normal, therefore any Hamiltonian diffeomorphism of the sphere admits such a decomposition.

In \cite{CGHMSS22}, they show that the $A$-fragmentation norm is unbounded on $\Ham(S^2)$ for any $A\in(0,\frac 1 2)$.

One can also define an autonomous norm for $\varphi\in\Ham(S^2)$:
\[\Vert \varphi\Vert_{\text{Aut}}=\min\{n,\exists \varphi_1,\ldots,\varphi_n\in\Aut(S^2), \varphi=\varphi_1\circ\ldots\circ\varphi_n\}.\]

This norm has been extensively studied on various surfaces \cite{GAMBAUDO_GHYS_2004,BKS,BK}. In particular, it is known to be unbounded on $\Ham(S^2)$ (\cite{GAMBAUDO_GHYS_2004}).
We define the following mixed fragmentation norm, for $\varphi\in\Ham(S^2)$:

\[\Vert \varphi\Vert_{\text{mix}}=\min\{n,\exists \varphi_1,\ldots,\varphi_n \text{ autonomous or supported in a disc of area}\leq \frac 1 2, \varphi=\varphi_1\circ\ldots\circ\varphi_n\}.\]

Note that this norm is smaller or equal to both the autonomous norm and the $\frac 1 2$-fragmentation norm.

Then, we show the following:
\begin{lemma}
    For any $\varphi\in\Ham(S^2(1))$,
\[\vert r_\lambda(\varphi)\vert \leq D\Vert \varphi\Vert_{\text{mix}}\]
where $D$ is the defect of $r_\lambda$. In particular, if $r_\lambda$ does not vanish identically, the mixed fragmentation norm and the $\frac 1 2$-fragmentation norm are unbounded on $\Ham(S^2)$.
\end{lemma}

\begin{proof}
    This follows from the fact that $r_\lambda$ vanishes on Hamiltonian diffeomorphisms that are autonomous or supported in a disc of area $\frac 1 2$. For autonomous Hamiltonian diffeomorphisms, this is the content of Remark \ref{rk:vanish_on_autonomous}. As for Hamiltonian diffeomorphisms supported inside a disc of area $\frac 1 2$, such a diffeomorphism is conjugated to an element of $\mathcal S(L_0)$ (by a Hamiltonian diffeomorphism mapping the disc to the northern hemisphere for instance), on which we know $r_\lambda$ vanishes (recall that $r_\lambda$ is conjugacy invariant since it is a homogeneous quasimorphism).
\end{proof}

\section{Annex}

In this annex, we prove Lemma \ref{lem.vanishing of partial} needed in the proof of Lemma \ref{lem.stretch L0 L2}.

\begin{lemmar}[Vanishing of $\partial$]\label{lem.vanishing of partial}
Let $A$ be an integral domain and $\partial \colon A^n \to A^n$ a homomorphism of $A$-module such that $\partial^2=0$ and $\text{rk}(\Ker(\partial)) =n$, then $\partial =0$.
\end{lemmar}

\begin{proof}
Let us assume that $\partial \neq 0$ for the sake of contradiction. As $A$ is integral, we can define $Q(A)$ the field of fraction of $A$, then
\[\tilde{\partial} \coloneqq \partial \otimes_A Q(A) \colon A^n \otimes_A Q(A) \to A^n \otimes_A Q(A)\]
is a linear map of vector spaces over the field $Q(A)$, $\tilde{\partial}^2=0$ and 
\[\text{rk}(\Ker(\partial))= \text{dim}(\Ker(\partial\otimes_A Q(A)))=\dim(\Ker(\tilde{\partial}))\leq n-1.\]
Otherwise, $\tilde{\partial}$ would be identically vanishing, this would imply the same for $\partial$.
\end{proof}

\bibliographystyle{alpha}
\bibliography{biblio}

\end{document}